\documentclass[12pt]{amsart}
\usepackage{fullpage}
\usepackage{amssymb}
\usepackage{commath}
\usepackage{graphicx}
\usepackage{epstopdf}
\usepackage{hyperref}
\usepackage{comment}
\newcommand{\R}{\mathbb{R}}
\newcommand{\Q}{\mathbb{Q}}
\newcommand{\C}{\mathbb{C}}
\newcommand{\Z}{\mathbb{Z}}

\newcommand{\bigoh}{{\mathcal{O}}}
\newcommand{\Nmin}{N^-}
\newcommand{\hmaj}{h^+}
\newcommand{\fmin}{f^-}
\newcommand{\fmaj}{f^+}
\newcommand{\Fmaj}{F^+}
\renewcommand{\H}{\mathbb{H}}
\DeclareMathOperator{\e}{e}
\DeclareMathOperator{\sinc}{sinc}
\DeclareMathOperator{\Si}{Si}
\DeclareMathOperator{\Tr}{Tr}
\DeclareMathOperator{\PSL}{PSL}
\DeclareMathOperator{\sgn}{sgn}

\newtheorem{theorem}{Theorem}[section]
\newtheorem{corollary}[theorem]{Corollary}
\newtheorem{lemma}[theorem]{Lemma}
\newtheorem{proposition}[theorem]{Proposition}
\theoremstyle{definition}
\newtheorem*{objective}{Objective}
\theoremstyle{remark}
\newtheorem*{remark}{Remark}
\newtheorem*{remarks}{Remarks}
\numberwithin{equation}{section}

\begin{document}
\title{Turing's method for the Selberg zeta-function}
\author{Andrew R. Booker}
\email{andrew.booker@bristol.ac.uk}

\author{David J. Platt}
\email{dave.platt@bris.ac.uk}
\address{School of Mathematics, University of Bristol, Bristol, BS8 1TW, UK}
\thanks{The authors were partially supported by EPSRC Grant \texttt{EP/K034383/1}.}

\date{}

\begin{abstract}
In one of his final research papers, Alan Turing introduced a method
to certify the completeness of a purported list of zeros of the Riemann
zeta-function.  In this paper we consider Turing's method in the analogous
setting of Selberg zeta-functions, and we demonstrate that it can be carried
out rigorously in the prototypical case of the modular surface.
\end{abstract}

\maketitle

\section{Introduction}
In \cite{turing}, Turing described and implemented a numerical procedure
for verifying the Riemann Hypothesis (RH) up to a given height
$T$ in the critical strip.  Turing's procedure was
similar to earlier numerical investigations of RH by Gram \cite{gram},
Backlund \cite{backlund}, Hutchinson \cite{hutchinson} and Titchmarsh
\cite{titchmarsh}, in that they were all based on isolating zeros on the
critical line by finding sign changes in the Hardy function $Z(t)$, and
then confirming that no zeros had been missed. Turing's approach differs
only in the latter step; where the earlier authors used ad hoc procedures
that are valid only for small values of $T$, Turing introduced a method
for certifying the completeness of a purported list of zeros of $Z(t)$
that is guaranteed to work (when the list is in fact complete). Turing's
method has remained a small but essential ingredient in all subsequent
verifications of RH and its many generalizations; see \cite{booker2006}
and \cite{booker2016} for more on Turing's method and its historical
background.

Meanwhile, researchers in the high energy physics community have
since the early 1990s applied the same idea to certifying lists of
zeros of Selberg zeta-functions for hyperbolic manifolds, albeit at a
heuristic level (without explicit error estimates) and without attribution
to Turing; see for instance \cite{steil}, where it was used in one of
the first investigations of large eigenvalues of the Laplacian for the
modular group, $\PSL(2,\Z)$.  In this paper we show, much in the
spirit of Turing's computations, that the method can be made rigorous
in the case of the modular group.

We begin by describing Turing's method in greater detail, in the context
of the Selberg zeta-function.
Let $\H=\{z\in\C:\Im(z)>0\}$ denote the hyperbolic plane, and
let $\{f_j\}_{j=1}^\infty$ be a complete sequence of Hecke--Maass
cuspforms on $\PSL(2,\Z)\backslash\H$, with Laplacian eigenvalues
$\frac14+r_j^2$ satisfying $0<r_1\le r_2\le\cdots$.  Then the associated
Selberg zeta-function $Z_{\PSL(2,\Z)}(s)$ has zeros at $s=\frac12\pm ir_j$.

Let $N(t)=\#\{j:r_j\le t\}$
denote the counting function of zeros in
the upper half plane, and suppose that we have accurately computed several
$r_j$ up to some height $T$, so we can construct a minorant $\Nmin(t)$
of the step function $N(t)$, for $t\le T$.  Suppose hypothetically that
we miss an $r_j$ at $T-H$ for some $H>0$, so that $N(t)\ge \Nmin(t)+1$
for $t\ge T-H$.  Integrating this inequality, we get
\begin{equation}
\int_0^T\Nmin(t)\dif t+H \le
\int_0^TN(t)\dif t = \int_0^T(T-t)\dif N(t)=\Tr(h_0),
\end{equation}
where $h_0(r)=\max(0,T-|r|)$
and $\Tr(h)=\sum_{j=1}^\infty h(r_j)$
denotes the trace of $h$ over the cuspidal spectrum.
Unfortunately, although $h_0$ has
a trace in this sense, it is not suitable for applying the Selberg trace
formula, but we can get around that by replacing $h_0$ by a majorant
$\hmaj_0$ with Fourier transform of compact support.  Thus, we have
$\int_0^T\Nmin(t)\dif t+H \le \Tr(\hmaj_0)$.

Applying the trace formula, the right-hand side will be the expected
main term (described by Weyl's law) plus the error that arises from
truncating the support of the Fourier transform.  If it happens that we
have not actually missed any zeros and we know them precisely enough
then we can expect $\int\Nmin(t)\dif t$ to be close to the main term,
and in fact it may even exceed the main term sometimes.  This gives us
an upper bound $H\le H^+$, i.e.\ we can provably show that there are
no missing zeros up to height $T-H^+$.  Moreover, using both upper and
lower bounds for $\Tr(h_0)$ we can estimate $\int_{T_1}^{T_2}N(t)\dif t$,
which would allow one to carry out the procedure using only the zeros
in an interval around $T$.

This approach is guaranteed to work for large enough $T$ because the
error term $S(t)$ in Weyl's law has mean value $0$; precisely, it is
known that $\frac1T\int_0^T S(t)\dif t=\bigoh\bigl((\log T)^{-2}\bigr)$
(see \cite[Ch.~10, Thm.~2.29]{hejhal2}). Consequently, one can
expect to prove by this method that there are no missing zeros up to
$T-\bigoh\bigl(T(\log T)^{-2}\bigr)$.

The remainder of the paper is devoted to obtaining such a bound with explicit
(and practical!) constants. Our precise result is the following.
\begin{theorem}\label{th:turing}
Define
$$
S(t) = N(t)-\left(\frac{t^2}{12}
-\frac{2t}{\pi}\log\frac{t}{e\sqrt{\frac{\pi}2}}-\frac{131}{144}
\right)
\quad\mbox{and}\quad
E(t)=
\left(1+\frac{6.59125}{\log t}\right)\!\left(\frac{\pi}{12\log t}\right)^2.
$$
Then for $T>1$,
\begin{equation}\label{ineq:St}
\frac1T\int_0^T S(t)\dif t\le E(T).
\end{equation}
\end{theorem}

To demonstrate the usefulness of this bound in practice, we applied
it to the list of zeros $r_j\le 178$, which are shown to 20 decimal
place accuracy at \cite{data}. This list was kindly provided to us by
Andreas Str\"ombergsson, who computed the $r_j$ using Hejhal's algorithm
\cite{hejhal} and certified them using the program from \cite{BSV}.
Using Theorem~\ref{th:turing}, we obtain the following:
\begin{corollary}\label{cor:andreas}
The Selberg zeta-function for $\PSL(2,\Z)\backslash\H$ has exactly
$2\,184$ zeros with imaginary part in $(0,177.75]$. All of them are
simple.
\end{corollary}

\begin{remarks}\
\begin{enumerate}
\item Our method can also be used to show the lower bound
\begin{equation*}
\frac1T\int_0^TS(t)\dif t\ge-\bigl(2+o(1)\bigr)
\left(\frac{\pi}{12\log t}\right)^2,
\end{equation*}
though we do not prove that here. The factor of $2$ difference
between the upper and lower bounds is due to an asymmetry when
approximating $h_0$ above and below by band-limited functions, as we
make explicit in the next section.
\item Following Selberg, Hejhal proved the analogue of the estimate
$\frac1T\int_0^T S(t)\dif{t}\ll(\log{T})^{-2}$ for a general cofinite
Fuchsian group, using the theory of the Selberg zeta-function; see
\cite[Ch.~2, Thm.~9.7]{hejhal1} and \cite[Ch.~10, Thm.~2.29]{hejhal2}.
Sarnak has suggested that this estimate could be obtained directly via
the trace formula, and our work realizes that goal in the case of
$\PSL(2,\Z)$. Our method could be generalized to congruence subgroups,
and we expect the implied constants that it produces to compare favorably
to the Selberg--Hejhal method (which has not been made explicit, to
our knowledge).

In the specific case of the modular group, the full asymptotic
for $N(t)$ appearing in Theorem~\ref{th:turing} was computed by Steil
\cite{steil}. A detailed proof was given by Jorgenson, Smajlovi\'c and
Then \cite{JST}, who also computed the lower-order terms of the
asymptotic in the case of moonshine groups.
\item The leading order constant $(\pi/12)^2$ could in principle be divided by
$4$ by using the Kuznetsov formula and the method of Li and Sarnak
\cite{li-sarnak}; however, it would be quite cumbersome to work out an
explicit error term in that setting.  A more practical means of achieving
a factor of $2$ savings in the asymptotic result would be to split the
spectrum into even and odd parts and derive a bound for each separately;
even there, however, the secondary error terms would be worse, and it
would likely not result in a savings for $T$ of any practical size.
\item The estimate \eqref{ineq:St} is substantially weaker than
Turing's estimate in the context of the Riemann zeta-function,
which is $\bigoh(\frac{\log{T}}{T})$.  One key reason for this is that the
Selberg zeta-function has a much higher density of zeros at large height
($\sim\frac16t$ vs.$~\sim\frac1{2\pi}\log{t}$ for Riemann zeta), which in
turn makes $S(t)$ noisier; see Figures~\ref{fig:St_sel} and \ref{fig:St_zeta} for a comparison
of the two over the range $t\in[0,100]$. In fact it is only by virtue
of the fact that the analogue of RH is known to hold
in this context that we can prove that $\frac1T\int_0^TS(t)\dif t=o(1)$
as $T\to\infty$.  The true rate of decay of $\frac1T\int_0^TS(t)\dif t$
is not known, but extensive numerics of Then \cite{then} suggest
that it should be $o(T^{-\frac12})$.
\end{enumerate}
\end{remarks}

\begin{figure}[h]
\centering
\fbox{\includegraphics[width=0.9\linewidth,height=0.22\linewidth]{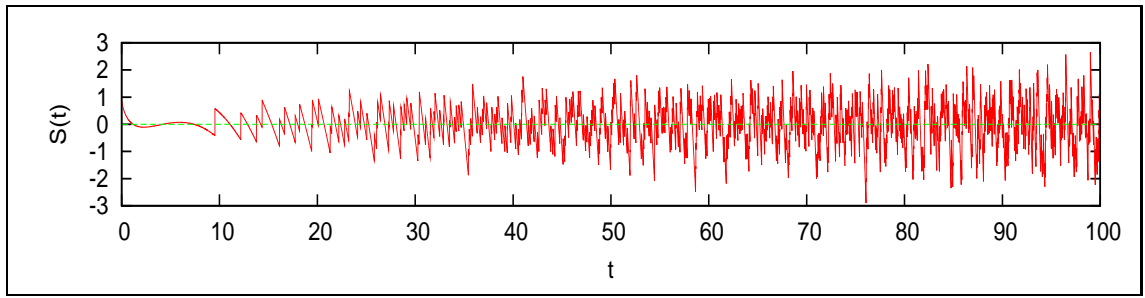}}
\caption{$S(t)$ for the Selberg zeta-function.}
\label{fig:St_sel}
\end{figure}

\begin{figure}[h]
\centering
\fbox{\includegraphics[width=0.9\linewidth,height=0.22\linewidth]{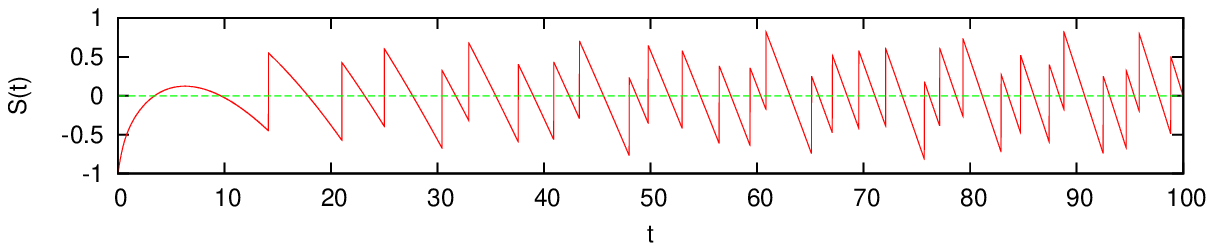}}
\caption{$S(t)$ for the Riemann zeta-function.}
\label{fig:St_zeta}
\end{figure}

The outline of the paper is as follows. In Section~\ref{sec:extremal},
we show how the optimization of our upper bound leads naturally
to an extremal problem in Fourier analysis, and we sketch a
solution that explains the leading-order constants that we can
expect to achieve.  In Section~\ref{sec:trace}, we recall the
Selberg trace formula for $\PSL(2,\Z)$, which will be our main
tool. We will then deduce an asymptotic for the main term of the
trace formula applied to $h_0$ (Section~\ref{sec:asym}) which we
can use to estimate the average of $S(t)$ for various ranges of $T$
(Section~\ref{sec:sint}). Section~\ref{sec:B} describes the computation
to bound the constant $B$ introduced in Section~\ref{sec:sint}, and
finally, in Section~\ref{sec:comp}, we verify Theorem~\ref{th:turing}
by considering $T$ in the four ranges $(1,100]$,
$[100,27\,400]$, $[27\,400,10^6]$ and $[10^6,\infty)$.
Appendix~\ref{app:comp} contains some details of the techniques used to
perform the necessary computations rigorously.

\subsection*{Acknowledgement}
We thank Peter Sarnak for helpful comments and the anonymous referees for
their valuable feedback and corrections.

\subsection{An extremal problem}\label{sec:extremal}
The main error in our estimate comes from approximating $h_0$
by $\hmaj_0$ for the upper bound, and similarly by a minorant
for the lower bound.  To control these errors, we have two contrary
objectives.  First, we want $\hmaj_0$ to be a good approximation to
$h_0$, in order to minimize the contribution from the main terms of
the trace formula, which for large $T$ are essentially of the form
$\frac1{12}\int_{\R}|r|\hmaj_0(r)\dif r$.  Second, we want the support
of the Fourier transform of $\hmaj_0$ to be as small as possible, in
order to control the contribution from the hyperbolic terms.  Thus,
we are led naturally to the following extremal problem.
\begin{objective}
Set $f(x)=\max(0,1-|x|)$.
Given a large $\Delta>0$, find
functions $f^\pm$ such that
\begin{enumerate}
\item the Fourier transforms $\hat{f}^\pm$
are supported in $[-\Delta,\Delta]$;
\item $\fmin(x)\le f(x)\le \fmaj(x)\quad\forall x\in\R$;
\item $\int_{\R}|x|(f(x)-\fmin(x))\dif x$ and
$\int_{\R}|x|(\fmaj(x)-f(x))\dif x$ are minimal.
\end{enumerate}
\end{objective}
In this section, we obtain the following result toward this objective.
\begin{proposition}
\label{extremalfunction}
There are functions $f^\pm$ as above with
\begin{equation}
\int_{\R}|x|(f(x)-\fmin(x))\dif x=\frac{1+o(1)}{6\Delta^2}
\quad\mbox{and}\quad
\int_{\R}|x|(\fmaj(x)-f(x))\dif x=\frac{1+o(1)}{12\Delta^2}
\end{equation}
as $\Delta\to\infty$.
\end{proposition}
\begin{remark}
It seems likely that this result is asymptotically best possible.
However, we do not attempt to prove that assertion here.
\end{remark}
\begin{proof}[Proof (sketch)]
Since the two cases are similar, we present the proof for $\fmaj$
only.
Let $\varphi:\R\to\R$ be an even function with Fourier transform supported
in $[-1,1]$, $\hat{\varphi}(0)=1$ and $\varphi(x)\ll \frac1{1+x^6}$.
Set $\varphi_{\Delta}(x) = \Delta\varphi(\Delta x)$.
Let $u(x)=\frac12(1+\sgn{x})$ be the unit step function, and
$v(x)=\max(0,x)$; these are approximated by the successive
integrals $U(x) = \int_{-\infty}^x\varphi(y)\dif y$ and $V(x) =
\int_{-\infty}^xU(y)\dif y$.  Set $E(x) = U(x)-u(x)$ and $F(x) =
V(x)-v(x)$.  Then we have
\begin{equation}
\begin{aligned}
f*\varphi_{\Delta}(x)-f(x)&=
-\int_{\R}f(t)dE(\Delta(x-t))
=\int_{\R}f'(t)E(\Delta(x-t))\dif t\\
&=\Delta^{-1}\bigl[F(\Delta(x-1))+F(\Delta(x+1))
-2F(\Delta x)\bigr].
\end{aligned}
\end{equation}
Now suppose that $F(x)\ge 0$ for all $x$.  Let
$\Fmaj$ be a majorant of $F$, with Fourier transform supported
in $[-1,1]$, such that $\Fmaj(x)-F(x)=\bigoh(x^{-4})$; such a function is
easily constructed using the techniques of \cite{vaaler}.  Then
$\fmaj(x)=f*\varphi_{\Delta}(x)+2\Delta^{-1}\Fmaj(\Delta x)$
majorizes $f(x)$.  Further, we have
\begin{equation}\label{extremalintegral}
\begin{aligned}
\int_{\R}&|x|(\fmaj(x)-f(x))\dif x
=\Delta^{-3}\int_{\R}|x|\bigl[F(x-\Delta)+F(x+\Delta)+
2(\Fmaj(x)-F(x))\bigr]\dif x\\
&=2\Delta^{-2}\int_{\R}F(x)\dif x
+2\Delta^{-3}\int_{|x|>\Delta}(|x|-\Delta)F(x)\dif x
+2\Delta^{-3}\int_{\R}|x|(\Fmaj(x)-F(x))\dif x\\
&=2\Delta^{-2}\int_{\R}F(x)\dif x
+\bigoh\bigl(\Delta^{-3}\bigr).
\end{aligned}
\end{equation}

Clearly we want to minimize $\int_{\R}F(x)\dif x$.  Equivalently,
$V$ should be a majorant of $v$ of exponential type $2\pi$
such that $\int_{\R}(V(x)-v(x))\dif x$ is minimal.
The unique such
function, found recently by Littmann \cite{Littmann2006}, is $V_0(x) =
\left(\frac{\cos\pi x}{\pi}\right)^2 \bigl[x\psi'(\tfrac12-x)+1\bigr]$,
where $\psi(x)=\frac{\Gamma'}{\Gamma}(x)$; it satisfies
$2\int_{\R}(V_0(x)-v(x))\dif x=\frac1{12}$.  Unfortunately, we cannot simply
set $V=V_0$ in our application, for then the integral over $|x|>\Delta$
in \eqref{extremalintegral} would diverge.
However, it is not hard to see that one can come arbitrarily close to
the constant $\frac1{12}$ using an approximation argument (e.g., we can
take $V$ as defined in Proposition~\ref{prop:varphi}, with $X=1$ and
$\delta$ sufficiently small); we carry out
the full details of this in \S\ref{sub:varphi}.
\end{proof}

\section{The trace formula for $\PSL(2,\Z)$}\label{sec:trace}

Our main tool will be the following version of the Selberg trace formula.
\begin{proposition}[The Selberg trace formula for Maass forms on
$\PSL(2,\Z)\backslash\H$]\label{prop:trace}
Let $\{r_j\}_{j=1}^\infty$ be as in the introduction.  Further, for a
fixed $\delta>0$, let $h(t)$ be an analytic function on $\{t\in\C:|\Im
t|\leq\frac12+\delta\}$ satisfying $|h(t)|\ll(1+|\Re t|)^{-2-\delta}$
and $h(t)=h(-t)$. Define $\hat{h}$, the Fourier transform of $h$, via
\begin{equation*}
\hat{h}(t)=\int_\R h(r)\exp(-2\pi irt)\dif r.
\end{equation*}
Then,
\begin{equation*}
\Tr(h)=\sum_{j=1}^\infty h(r_j)=M(h)+R(h),
\end{equation*}
where
\begin{equation*}
M(h)=(I+E+P)(h)-h(0),\quad
R(h)=(D-C)(\hat{h}),
\end{equation*}
\begin{align*}
I(h)&=\frac{1}{12}\int_\R r\tanh(\pi r)h(r)\dif 
r=-\frac{1}{24\pi}\int_\R\frac{\hat{h}'(t)}{\sinh \pi t}\dif t,\\
E(h)&=\int_\R\frac{\frac{1}{8}+\frac{1}{3\sqrt{3}}\cosh\!\left(\frac{\pi r}{3}\right)}{\cosh(\pi r)}h(r)\dif r
=\int_\R\left(\frac{1}{8\cosh\pi t}+\frac{2\cosh\pi t}{3+6\cosh 2\pi t}
\right)\hat{h}(t)\dif t,\\
P(h)&=\frac{1}{2\pi}\int_\R(\log 2\pi -2\psi(1+2ir))h(r)\dif r\\
&=\frac{\hat{h}(0)}{2\pi}\left(\log\frac{\pi}{2}+2\gamma\right)-\frac{h(0)}{4}-\frac{1}{\pi}\int_0^\infty\log\!\left(4\sinh\!\left(\frac{\pi
t}{2}\right)\right)\hat{h}'(t)\dif t,
\end{align*}
\begin{align*}
D(\hat{h})=
\frac1\pi&\sum_{t=3}^\infty \frac{L(1,\chi_d)}{l}
\prod_{p\mid l}\left(1+\left(p-\chi_d(p)\right)
\frac{\left(p^\infty,l\right)-1}{p-1}\right)
\hat{h}\!\left(\frac{1}{\pi}\log\!\left(\frac{t+\sqrt{t^2-4}}{2}\right)\right)\\
&+\sum_{n=1}^\infty\frac{\Lambda(n)}{n}\hat{h}\!\left(\frac{\log n}{\pi}\right)
\end{align*}
and
\begin{equation*}
C(\hat{h})=\int_\R\hat{h}(t)\left(\cosh \pi t-1\right)\dif t.
\end{equation*}
In the definition of $D$, we write $t^2-4=dl^2$, where $d$ is a
fundamental discriminant, $l>0$, $\chi_d$ is the corresponding quadratic
character, and $(p^\infty,l)=\gcd(p^\infty,l)$ is the largest power of
$p$ dividing $l$.
\end{proposition}
\begin{remark}
We refer to $I$, $E$, $P$, $D$ and $C$ as the identity, elliptic,
parabolic, discrete and continuous terms, respectively.
For ``wide'' test functions (those of the sort used to measure Weyl's
Law), one can think of $M$ and $R$ as the main term and remainder,
respectively.

Although the terms of the trace formula are only defined for analytic
test functions, both $\Tr(h)$ and $M(h)$ can be interpreted for any
even, continuous function $h:\R\to\C$ satisfying
$h(t)\ll(1+|t|)^{-2-\delta}$. In turn, for any such $h$, we may define
$R(h)$ by the equality $\Tr(h)=M(h)+R(h)$.
\end{remark}

\begin{proof}
Theorem $3$ of \cite{BookerandStrombergsson2007} gives the trace formula
for $\Gamma_0^{\pm}(N)$ with character $\chi$. Thus we need to set
$N=1$, take $\chi$ to be the trivial character of modulus $q(\chi)=1$
and then sum over $\epsilon\in\{0,1\}$ to account for both even and
odd eigenfunctions. We also define $\chi_d$ and $t^2-4=dl^2$ as in the
statement of the theorem and $d(n)$ to be the usual divisor
function. Now, considering the term
\begin{equation*}
\frac{\prod_{p\mid N}(p+1)}{24}\int_{\R}rh(r)\tanh(r)\dif r,
\end{equation*}
we note that the product is empty so $\epsilon\in\{0,1\}$ together contribute
\begin{equation*}
\frac{1}{12}\int_{\R}rh(r)\tanh(r)\dif r = I(h).
\end{equation*}
Turning now to
\begin{equation*}
-\frac{d(N)}{4}g(0)\left(\log(8N)
+C_{\chi,\epsilon}\log\!\left(\frac{Nq(\chi)}{2\pi^2}\right)
+\frac{1}{2}(C_{\chi,\epsilon}-1)\log\gcd(N,2)\right),
\end{equation*}
since $\chi(-1)=1$ (or, in the language of
\cite{BookerandStrombergsson2007}, $\chi$ is ``pure''), we have
$C_{\chi,\epsilon}=2,0$ for $\epsilon=0,1$, respectively. Thus we get
\begin{equation*}
-\frac{1}{2}\frac{\hat{h}(0)}{2\pi}(3\log 2-\log 2-2\log\pi)
=-\frac{\hat{h}(0)}{2\pi}\log\frac{2}{\pi}\\
=\frac{1}{2\pi}\int_{\R}\log\!\left(\frac{\pi}{2}\right)h(r)\dif r.
\end{equation*}
The next line,
\begin{equation*}
-\frac{d(N)}{4\pi}\int_{\R}h(r)\left[\psi\!\left(\frac{1}{2}+ir\right)
+C_{\chi,\epsilon}\psi(1+ir)\right]\dif r,
\end{equation*}
yields
\begin{equation*}
-\frac{1}{2\pi}\int_{\R}h(r)\left[2\psi(1+2ir)-2\log 2\right]\dif r,
\end{equation*}
and combining this with the previous term we get
\begin{equation*}
\frac{1}{2\pi}\int_{\R}\left[\log(2\pi)-2\psi(1+2ir)\right]h(r)\dif r=P(h).
\end{equation*}

The final line contributes nothing as the first sum therein is empty and the first term of the penultimate line does not apply either as $\chi=1$. The second term, i.e.\
\begin{equation*}
d(n)\sum_{n=1}^\infty\frac{\Lambda(n)\cdot\{\chi\}_\epsilon(n)}{n}g(2\log n),
\end{equation*}
results in
\begin{equation}\label{eq:D1}
\frac{1}{\pi}\sum_{n=1}^\infty\frac{\Lambda(n)}{n}\hat{h}\!\left(\frac{\log n}{\pi}\right).
\end{equation}
This leaves us with the term spanning lines 2 and 3. \cite[(2.60)]{BookerandStrombergsson2007} gives us
\begin{equation*}
\langle\chi(\delta)\rangle_{\delta^2-t\delta+n\equiv0}=1
\end{equation*}
and the argument on page 142 of the same shows that
\begin{equation*}
\sum_{f\mid l}\textbf{h}^+(\mathfrak{r}[f])
\cdot[\mathfrak{r}[l]^l:\mathfrak{r}[f]^l]=
\textbf{h}^+(d)\prod_{p\mid l}
\left(1+(p-\chi_d(p))\frac{(p^\infty,l)-1}{p-1}\right),
\end{equation*}
with $\textbf{h}^+(d)$ denoting the narrow class number of $\mathbb{Q}(\sqrt{d})$. This leads us to  
\begin{equation*}
\sum_{\substack{t\in\Z\\\sqrt{t^2-4N}\notin\Q}}\textbf{h}^+(d)
\prod_{p\mid l}\left(1+(p-\chi_d(p))\frac{(p^\infty,l)-1}{p-1}\right) A(t,1),
\end{equation*}
where 
\begin{equation*}
\begin{aligned}
A(t,n)=
\begin{cases}
\frac{\log\epsilon_1}{\sqrt{t^2-4n}}\cdot g\!\left(\log\frac{(|t|+\sqrt{t^2-4n})^2}{4}\right) & \text{if  $t^2-4n>0$,}\\
\frac{2}{|\mathfrak{r}[1]^1|\cdot\sqrt{4-t^2}}\cdot\int_{\R}\frac{\exp(-2r\cdot\arccos(|t|/2))}{1-\exp(-2\pi r)}h(r)\dif r & \text{if $t^2-4n<0$.}
\end{cases}
\end{aligned}
\end{equation*}
Here $\epsilon_1$ is the proper fundamental unit in
$\Z[\frac{d+\sqrt{d}}2]$ and $|\mathfrak{r}[1]^1|$ is the size of its norm one
unit group. Considering the sum for $|t|\geq3$, where $t^2-4>0$, we get
\begin{equation}\label{eq:D2}
\begin{aligned}
&\frac{1}{\pi}\sum_{t=3}^\infty\textbf{h}^+(d)\prod_{p\mid
l}\left(1+(p-\chi_d(p))\frac{(p^\infty,l)-1}{p-1}\right) \frac{\log
\epsilon_1}{\sqrt{t^2-4}}\hat{h}\!\left(\frac{1}{\pi}\log\!\left(\frac{t+\sqrt{t^2-4}}{2}\right)\right)\\
&=\frac{1}{\pi}\sum_{t=3}^\infty \frac{L(1,\chi_d)}{l}\prod_{p\mid
l}\left(1+(p-\chi_d(p))\frac{(p^\infty,l)-1}{p-1}\right)
\hat{h}\!\left(\frac{1}{\pi}\log\!\left(\frac{t+\sqrt{t^2-4}}{2}\right)\right),
\end{aligned}
\end{equation}
by Dirichlet's class number formula.
We now combine \eqref{eq:D1} and \eqref{eq:D2} to get $D(\hat{h})$.
Turning to the $t=0$ term, we have $d=-4$, $l=1$, $\textbf{h}^+(-4)=1$,
$|\mathfrak{r}[1]^1|=4$ and we get
\begin{equation*}
\textbf{h}^+(-4)A(0,1)
=\textbf{h}^+(-4)\frac{2}{|\mathfrak{r}[1]^1|\cdot\sqrt{4}}
\int_{\R}\frac{\exp(-\pi r)}{1+\exp(-2\pi r)}h(r)\dif r\\
=\int_{\R}\frac{1}{8\cosh(\pi r)} h(r) \dif r.
\end{equation*}
Now we consider the $t=\pm1$ terms ($d=-3$, $l=1$, $\textbf{h}^+(-3)=1$,
$|\mathfrak{r}[1]^1|=6$), which contribute
\begin{equation*}
\begin{aligned}
\textbf{h}^+(-3)\frac{4}{|\mathfrak{r}[1]^1|\cdot\sqrt{3}}
\int_{\R}\frac{\exp\!\left(-\frac{2\pi r}{3}\right)}{1+\exp(-2\pi r)}h(r)\dif r
=\frac{1}{3\sqrt{3}}\int_{\R}\frac{\cosh\!\left(\frac{\pi r}{3}\right)}
{\cosh(\pi r)}h(r)\dif r,
\end{aligned}
\end{equation*}
where we have used the fact that $h$ is even. Thus the $A(\pm 1,1)$ terms together with the $A(0,1)$ term make up $E(h)$. Clearly the $t=2$ term contributes nothing as $\sqrt{4-4}\in\Q$.

Finally, we have the contribution from the constant function with 
eigenvalue $\lambda=\frac14+r^2=0$, which leads to
\begin{equation*}
-h\!\left(\frac{i}{2}\right)=-\int_{\R}\hat{h}(t)\cosh \pi t \dif t,
\end{equation*}
since $\hat{h}$ is even. We can now write
\begin{equation*}
-\int_{\R}\hat{h}(t)\cosh \pi t\dif t
=-\int_{\R}\hat{h}(t)\bigl[\cosh \pi t-1\bigr]\dif t-h(0)
=-C(\hat{h})-h(0),
\end{equation*}
and we have the first form of the trace formula as stated.

We now use the results of \cite[p.~141]{BookerandStrombergsson2007}
(adjusted for our definition of the Fourier transform) to convert from
$h$ to $\hat{h}$, to whit:
\begin{comment}
\begin{equation*}
\begin{aligned}
h(0)&=\int_{\R} \hat{h}(t)\dif t,\\
\int_{\R}rh(r)\tanh(\pi r)\dif r &= -\frac{1}{2\pi}\int_{\R}\frac{\hat{h}'(t)}{\sinh(\pi t)}\dif t,\\
\int_{\R}\frac{h(r)}{\cosh(\pi r)}\dif r &= \int_{\R}\frac{\hat{h}(t)}{\cosh(\pi t)}\dif t\textrm{ and}\\
\int_{\R}\frac{\cosh\!\left(\frac{\pi r}{3}\right)}{\cosh(\pi r)}h(r)\dif r &=2\sqrt{3}\int_{\R}\frac{\cosh \pi t}{1+2\cosh 2\pi t}\hat{h}(t)\dif t.
\end{aligned}
\end{equation*}
\end{comment}
\begin{equation*}
\begin{aligned}
h(0)&=\int_{\R} \hat{h}(t)\dif t,\\
\int_{\R}rh(r)\tanh(\pi r)\dif r &= -\frac{1}{2\pi}\int_{\R}\frac{\hat{h}'(t)}{\sinh(\pi t)}\dif t\textrm{ and}\\
\int_{\R}\frac{h(r)}{\cosh(\pi r)}\dif r &= \int_{\R}\frac{\hat{h}(t)}{\cosh(\pi t)}\dif t.\\
\end{aligned}
\end{equation*}
Referring to \cite[(12) on p.~31]{Batemanv1}, we also have
\begin{equation*}
\begin{aligned}
\int_{\R}\frac{\cosh\!\left(\frac{\pi r}{3}\right)}{\cosh(\pi r)}h(r)\dif r &=2\sqrt{3}\int_{\R}\frac{\cosh \pi t}{1+2\cosh 2\pi t}\hat{h}(t)\dif t.
\end{aligned}
\end{equation*}

Finally, we consider $P(h)$.
Adjusting the identities in \cite[p.~141]{BookerandStrombergsson2007} for
our Fourier transform convention, we have
\begin{align*}
\frac{1}{2\pi}\int_\R\psi(1+ir)h(r)\dif r &=
-\gamma\frac{\hat{h}(0)}{2\pi}+\frac{1}{4\pi^2}\int_0^\infty\log(u)\hat{h}'(u/2\pi)\dif u\\
&\hspace{1cm}+\frac{1}{8\pi}\int_\R\hat{h}(u/2\pi)\dif
u+\frac{1}{4\pi^2}\int_0^\infty\log\!\left(\frac{\sinh(u/2)}{u/2}\right)\hat{h}'(u/2\pi)\dif u
\end{align*}
and
\begin{align*}
\frac{1}{2\pi}&\int_\R\psi\!\left(\frac{1}{2}+ir\right)h(r)\dif r \\
&=-\gamma\frac{\hat{h}(0)}{2\pi}+\frac{1}{4\pi^2}\int_0^\infty\log(u)\hat{h}'(u/2\pi)\dif u\\
&\quad-\frac{1}{4\pi^2}
\int_0^\infty\log\!\left(\frac{\sinh(u/2)}{u/2}\right)\hat{h}'(u/2\pi)\dif u
+\frac{1}{2\pi^2}\int_0^\infty\log\!\left(\frac{\sinh(u/4)}{u/4}\right)\hat{h}'(u/2\pi)\dif u.
\end{align*}
Thus, starting from
\begin{align*}
P(h)&=\frac{1}{2\pi}\int_\R\left(\log(2\pi)-2\psi(1+2ir)\right)h(r)\dif r\\
&=\frac{1}{2\pi}\int_\R\left(\log(2\pi)-\psi(1+ir)-\psi(1/2+ir)-2\log 2\right)h(r)\dif r,
\end{align*}
we obtain
\begin{align*}
P(h)&=\frac{\hat{h}(0)}{2\pi}\left(\log\frac{\pi}{2}+2\gamma\right)-\frac{h(0)}{4}-2\int_0^\infty\log(4\sinh(u/4))\frac{\hat{h}'(u/2\pi)}{4\pi^2}\dif u\\
&=\frac{\hat{h}(0)}{2\pi}\left(\log\frac{\pi}{2}+2\gamma\right)-\frac{h(0)}{4}-\frac{1}{\pi}\int_0^\infty\log(4\sinh(\pi t/2))\hat{h}'(t)\dif t,%\\
%&=\frac{\hat{h}(0)}{2\pi}\left(\log(2\pi) +2\gamma\right)-\frac{1}{\pi}\int_0^\infty\log(1-\exp(-\pi t))\hat{h}'(t)\dif t,
\end{align*}
as required.
\end{proof}

\section{An asymptotic for $M(h_0)$.}\label{sec:asym}
Recall that if we set $h_0(r)=\max(0,T-|r|)$ then
$\Tr(h_0)=\int_0^T N(t)\dif t$, where
$N$ is the counting function of Theorem~\ref{th:turing}. Motivated by
this, we establish the following asymptotic upper bound for $M(h_0)$:
\begin{proposition}\label{prop:M}
For $T\geq 4$, we have
\begin{equation*}
M(h_0)\leq\frac{1}{36}T^3-\frac{\log(T)}{\pi}T^2
+\frac{3+\log\frac{\pi}{2}}{2\pi}T^2-\frac{131}{144}T+\frac{\log T}{24\pi}+C_0,
\end{equation*}
where
\begin{equation*}
C_0=\frac{\zeta(3)}{16\pi^3}-\frac{1}{4\pi^2}\left(2L(2,\chi_{-4})+3\sqrt{3}L(2,\chi_{-3})\right)-\frac{1}{2\pi}\left(\zeta'(-1)-\frac{\log 2 +1}{12}\right).
\end{equation*}
\end{proposition}

The proof proceeds by examining the contributions from the identity
term, the elliptic terms, the parabolic terms and the constant
eigenfunction in turn, as follows.
\subsection{The identity term}
\begin{lemma}\label{lem:MI}
For $T\ge0$, we have
\begin{equation*}
\frac{1}{12}\int_{-T}^T(T-|r|)r\tanh(\pi r)\dif r
\le\frac{T^3}{36}-\frac{T}{144}+\frac{\zeta(3)}{16\pi^3}.
\end{equation*}
\end{lemma}
\begin{proof}
We have
\begin{equation*}
\begin{aligned}
&\frac{1}{12}\int_{-T}^T(T-|r|)r\tanh(\pi r)\dif r\\
&=\frac{1}{6}\int_0^T(T-r)r\dif r+\frac{1}{6}\int_0^\infty (r-T)r(1-\tanh(\pi r))\dif r
-\frac{1}{6}\int_T^\infty(r-T)r(1-\tanh(\pi r))\dif r\\
&\le\frac{T^3}{36}+\frac{1}{6}\int_0^\infty (r-T)r(1-\tanh(\pi r))\dif r
=\frac{T^3}{36}-\frac{T}{144}+\frac{\zeta(3)}{16\pi^3}.
\end{aligned}
\end{equation*}
\end{proof}

\subsection{The elliptic terms}
\begin{lemma}\label{lem:ME}
Similarly, for $T\geq 0$ we get
\begin{align*}
\int_{-T}^T(T-|r|)&\frac{\frac{1}{8}+\frac{\sqrt{3}}{9}\cosh(\pi r/3)}
{\cosh(\pi r)}\dif r\le\frac{25T}{72}-\frac{L(2,\chi_{-4})}{2\pi^2}
-\frac{3\sqrt{3}L(2,\chi_{-3})}{4\pi^2}+E_e,
\end{align*}
where
\begin{align*}
E_e=\frac{\exp(-\pi T)}{2\pi}\left(T+\frac{1}{\pi}+\frac{\sqrt{3}}{12}\left[\e^{\pi T/3}\left(8T+\frac{12}{\pi}\right)+\e^{-\pi T/3}\left(4T+\frac{3}{\pi}\right)\right]\right).
\end{align*}
\end{lemma}
\begin{proof}
We have
\begin{align*}
\int_0^\infty\frac{\dif r}{\cosh(\pi r)}=\frac{1}{2},
&\quad\int_0^\infty\frac{r\dif r}{\cosh(\pi r)}
=\frac{2L(2,\chi_{-4})}{\pi^2},\\
\int_0^\infty\frac{\cosh(\pi r/3)\dif r}{\cosh(\pi r)}=\frac{1}{\sqrt{3}}
&\quad\text{and}\quad
\int_0^\infty\frac{r\cosh(\pi r/3)\dif r}{\cosh(\pi r)}
=\frac{27L(2,\chi_{-3})}{8\pi^2}.
\end{align*}
Thus we can write
\begin{align*}
&\int_{-T}^T(T-|r|)\frac{\frac{1}{8}+\frac{\sqrt{3}}{9}\cosh(\pi r/3)}{\cosh(\pi r)}\dif r\\
&=\frac{T}{4}\int_0^T\frac{\dif r}{\cosh(\pi r)}-\frac{1}{4}\int_0^T\frac{r}{\cosh(\pi r)}\dif r
+\frac{2\sqrt{3}T}{9}\int_0^T\frac{\cosh(\pi r/3)}{\cosh(\pi r)}\dif r -\frac{2\sqrt{3}}{9}\int_0^T\frac{r\cosh(\pi r/3)}{\cosh(\pi r)} \dif r\\
&\le\frac{T}{4}\int_0^\infty\frac{\dif r}{\cosh(\pi r)}-\frac{1}{4}\int_0^\infty\frac{r}{\cosh(\pi r)}\dif r+\frac{1}{4}\int_T^\infty\frac{r}{\cosh(\pi r)}\dif r
+\frac{2\sqrt{3}T}{9}\int_0^\infty\frac{\cosh(\pi r/3)}{\cosh(\pi r)}\dif r\\
&\hspace{1cm}-\frac{2\sqrt{3}}{9}\int_0^\infty\frac{r\cosh(\pi r/3)}{\cosh(\pi r)} \dif r+\frac{2\sqrt{3}}{9}\int_T^\infty\frac{r\cosh(\pi r/3)}{\cosh(\pi r)} \dif r\\
&\le\frac{T}{8}-\frac{L(2,\chi_{-4})}{2\pi^2}+\frac{\e^{-\pi T}}{2\pi}\left(T+\frac{1}{\pi}\right)+\frac{2T}{9}-\frac{3\sqrt{3}}{4\pi^2}L(2,\chi_{-3})\\
&\hspace{1cm}+\frac{\e^{-\pi T}}{2\pi}\cdot\frac{\sqrt{3}}{12}\left[\e^{\pi
T/3}\left(8T+\frac{12}{\pi}\right)+\e^{-\pi
T/3}\left(4T+\frac{3}{\pi}\right)\right].
\end{align*}
\end{proof}

\subsection{The parabolic terms}

We will need the following preparatory lemma.
\begin{lemma}\label{lem:rz}
Write $\log\Gamma(z)=A(z)+R(z)$, where
$A(z)=(z-\frac12)\log z-z+\frac12\log(2\pi)+1/(12z)$. Then we have
\begin{equation*}
C_{\frac{1}{2}}:=\Re\int_{\frac12}^{\frac12+i\infty}R(z)\dif z = \frac{1}{2}\zeta'(-1)+\frac{\log 2}{12}+\frac{1}{48}
\end{equation*}
and
\begin{equation*}
C_1:=\Re\int_{1}^{1+i\infty}R(z)\dif z = -\zeta'(-1)-\frac{1}{6}.
\end{equation*}
\end{lemma}
\begin{proof}
\begin{align*}
C_1&=\Re\int_1^{1+i\infty}\bigl(\log\Gamma(z)-A(z)\bigr)\dif z\\
&=2\Re\int_{\frac{1}{2}}^{\frac{1}{2}+i\infty}\bigl(
\log\Gamma(2z)-A(2z)\bigr)\dif z\\
&=2\Re\int_{\frac{1}{2}}^{\frac{1}{2}+i\infty}\bigl(
\log\Gamma(z)+\log\Gamma(z+\tfrac{1}{2})
-\tfrac12\log\pi+(2z-1)\log 2-A(2z)\bigr)\dif z\\
&=2\Re\int_{\frac{1}{2}}^{\frac{1}{2}+i\infty}\bigl(
A(z)+\log\Gamma(z+\tfrac{1}{2})-\tfrac12\log\pi
+(2z-1)\log 2-A(2z)\bigr)\dif z+2C_{\frac{1}{2}}\\
&=2\Re\int_1^{1+i\infty}\bigl(
A(z-\tfrac{1}{2})-A(2z-1)+\log\Gamma(z)-\tfrac12\log\pi
+(2z-2)\log 2\bigr)\dif z +2C_{\frac{1}{2}}.
\end{align*}
This leads us to
\begin{align*}
2C_{\frac{1}{2}}+C_1
&=2\Re\int_1^{1+i\infty}\bigl(
A(2z-1)-A(z)-A(z-\tfrac{1}{2})+\tfrac12\log\pi+(2-2z)\log 2\bigr)\dif z\\
&=-2\Im\int_0^\infty\bigl(A(1+it)-A(1+2it)+A(\tfrac{1}{2}+it)
-\tfrac12\log\pi+(1+2it)\log 2\bigr)\dif t,
\end{align*}
which, by the miracle that is Maple\textsuperscript{TM}, gives
\begin{equation*}
2C_{\frac{1}{2}}+C_1=\frac{\log 2}{6}-\frac{1}{8}.
\end{equation*}
In addition, by \cite[(A.14)]{Voros1987} we have
\begin{equation*}
\int_{\frac{1}{2}}^1\log\Gamma(x)\dif x
=\frac{\log\pi}{4}-\frac{1}{8}+\frac{7}{24}\log2+\frac{3}{2}\zeta'(-1),
\end{equation*}
so that
\begin{equation*}
C_{\frac{1}{2}}=C_1+\int_{\frac{1}{2}}^1 R(x)\dif x
=C_1+\int_{\frac{1}{2}}^1\bigl(\log\Gamma(x)-A(x)\bigr)\dif x
=C_1+\frac{3}{16}+\frac{\log 2}{12}+\frac{3}{2}\zeta'(-1),
\end{equation*}
and we are done.
\end{proof}

We can now handle the parabolic term as follows:
\begin{lemma}\label{lem:MP}
Let $T>1$. Then
\begin{align*}
P(h_0)&=\frac{1}{2\pi}\int_{-T}^T(T-|r|)(\log(2\pi)-\psi(1+2ir))\dif r\\
&\le\frac{3-2\log T+\log(\pi/2)}{2\pi}T^2-\frac{T}{4}+\frac{1}{24\pi}+\frac{\log T}{24\pi}+\frac{\log 2}{24\pi}-\frac{\zeta'(-1)}{2\pi}\\
&\hspace{1cm}-\frac{241}{5760\pi}T^{-2}+\frac{17641}{161280\pi}T^{-4}.
\end{align*}
\end{lemma}
\begin{proof}
We write
\begin{equation*}
\begin{aligned}
&\frac{1}{2\pi}\int_{-T}^T(T-|r|)(\log(2\pi)-\psi(1+2ir))\dif r\\
&=\frac{1}{2\pi}\int_{-T}^T(T-|r|)(\log(\pi/2)-\psi(1+ir)-\psi(\tfrac12+ir))\dif r\\
&=\log(\pi/2)T^2-\frac{1}{2\pi}\int_{-T}^T(T-|r|)(\psi(1+ir)+\psi(\tfrac12+ir))\dif r\\
&=\log(\pi/2)T^2+\frac{1}{\pi}\Re\int_{1}^{1+iT}\log\Gamma(z)\dif z
+\frac{1}{\pi}\Re\int_{\frac12}^{\frac12+iT}\log\Gamma(z)\dif z.
\end{aligned}
\end{equation*}
We now write $\log\Gamma(z)=A(z)+R(z)$ as in Lemma~\ref{lem:rz}. Then we have
\begin{equation*}
\begin{aligned}
\frac{1}{\pi}\Re\int_{\sigma}^{\sigma+iT}\log\Gamma(z)\dif z &= \frac{1}{\pi}\Re\left[\int_\sigma^{\sigma+iT}A(z)\dif z - \int_{\sigma+iT}^{\sigma+i\infty}R(z)\dif z + \int_\sigma^{\sigma+i\infty}R(z) \dif z\right].\\
\end{aligned}
\end{equation*}
The integrals involving $A$ evaluate to
\begin{align*}
\frac{1}{24\pi}\Bigl[(1-6T^2)\log(T^2+1)+36T^2-(12T^2+1)\log\sqrt{4T^2+1}
-12T\arctan(T)+12T^2\log 2\Bigr],
\end{align*}
and since $T>1$ we have
\begin{align*}
-\arctan(T)\leq -\frac{\pi}{2}+\frac{1}{T}-\frac{1}{3T^3}+\frac{1}{5T^5},
\end{align*}
\begin{align*}
\log(T^2+1)\leq 2\log T,
\end{align*}
\begin{align*}
-\log(T^2+1)\leq -2\log T-T^{-2}+\frac{1}{2}T^{-4}
\end{align*}
and
\begin{align*}
-\log\sqrt{4T^2+1}\leq -\log 2 - \log T-\frac18 T^{-2}+\frac{1}{32} T^{-4}
\end{align*}
so we can maximise the contribution from $A$ with
\begin{equation*}
\begin{aligned}
\frac{1}{\pi}&\left[\left(\frac{3}{2}-\log T\right)T^2-\frac{\pi}{4}T+\frac{3}{16}-\frac{\log 2}{24}+\frac{\log T}{24}-\frac{5}{128}T^{-2}+\frac{773}{7680}T^{-4}\right].
\end{aligned}
\end{equation*}
We can now use Lemma \ref{lem:rz} to handle the integrals of $R$ from $\sigma$ to $\sigma+i\infty$ and then using \cite[(4.1)]{Hare1997} to bound the error in Stirling's approximation we have
\begin{equation*}
R(z)\leq-\frac{1}{360z^3}+\frac{2}{315 |z|^5}
\end{equation*}
so that 
\begin{align*}
-\Re\left[\int_{1+iT}^{1+i\infty}R(z)\dif z
+\int_{\frac12+iT}^{\frac12+i\infty}R(z)\dif z\right]
&\leq -\Re \left[(1+iT)^{-2}+\left(\frac12 + iT\right)^{-2}\right]+\frac{1}{315}T^{-4}\\
&\leq -\frac{1}{360}T^{-2}+\frac{1}{180}T^{-4}+\frac{1}{315}T^{-4}.
\end{align*}

\end{proof}

\subsection{The constant eigenfunction}

\begin{lemma}\label{lem:MC}
The contribution from the constant eigenfunction is trivially
\begin{equation*}
-h_0(0)=-T.
\end{equation*}
\end{lemma}

\subsection{Proof of Proposition~\ref{prop:M}}
We now combine Lemmas~\ref{lem:MI}, \ref{lem:ME}, \ref{lem:MP} and
\ref{lem:MC}, observing that
\begin{equation*}
-\frac{241}{5760\pi}T^{-2}+\frac{17641}{161280\pi}T^{-4}
+E_e(T)
%\frac{\exp(-\pi T)}{2\pi}\left[T+\frac{1}{\pi}+\frac{\sqrt{3}}{12}\left\{\e^{\pi T/3}\left(8T+\frac{12}{\pi}\right)+\e^{-\pi T/3}\left(4T+\frac{3}{\pi}\right)\right\}\right]
\end{equation*}
is negative for $T\geq 4$.

\section{An upper bound on $\int S(t)\dif t$.}\label{sec:sint}
We will now use the trace formula to derive an upper bound for $\int S(t)\dif t$. We start with some preparatory lemmas.

\begin{lemma}\label{lem:fhat}
Let $\varphi:\R\to\R$ be a smooth, even function
such that $x^2\varphi(x)$ is
absolutely integrable and $\int_\R\varphi(x)\dif{x}=1$.
Define
\begin{equation*}
F(r)=\int_{-\infty}^r \int_{-\infty}^y \varphi(x) \dif x \dif y-\max(0,r).
\end{equation*}
Then $F$ is even and absolutely integrable, with Fourier transform
\begin{equation*}
\hat{F}(t)=\frac{1-\hat{\varphi}(t)}{(2\pi t)^2}.
\end{equation*}
\end{lemma}
\begin{proof}
Define
$$
U(y)=\int_{-\infty}^y\varphi(x)\dif{x}
\quad\text{and}\quad
u(y)=\frac{1+\sgn{y}}2,
$$
so that
$$
F(r)=\int_{-\infty}^r\bigl(U(y)-u(y)\bigr)\dif{y}.
$$
Then $U(y)-u(y)$ is an odd function of $y$, and
$$
\bigl|U(y)-u(y)\bigr|=\left|\int_{-\infty}^{-|y|}\varphi(x)\dif{x}\right|
\le y^{-2}\int_{-\infty}^0x^2|\varphi(x)|\dif{x}\ll y^{-2}.
$$
Hence, for $r\le0$, we have
$$
\int_{-\infty}^r(r-x)\varphi(x)\dif{x}
=\int_{-\infty}^rU(y)\dif{y}=F(r),
$$
by partial integration. Thus,
\begin{align*}
\int_{-\infty}^0|F(r)|\dif{r}
&\le\int_{-\infty}^0\int_{-\infty}^r(r-x)|\varphi(x)|\dif{x}\dif{r}
=\int_{-\infty}^0\int_x^0(r-x)|\varphi(x)|\dif{r}\dif{x}\\
&=\frac12\int_{-\infty}^0x^2|\varphi(x)|\dif{x}<\infty.
\end{align*}
Since $U-u$ is odd, $F$ is even, and therefore absolutely
integrable, by the above. Hence, $F$ has a continuous Fourier transform.

Next, let $f:\R\to\R_{\ge0}$ be a smooth, even function with mass $1$ and 
support $[-1,1]$.  For $\varepsilon>0$, define
$f_\varepsilon(x)=\varepsilon^{-1}f(\varepsilon^{-1}x)$
and
\begin{equation}\label{eq:Fepsdef}
F_\varepsilon(r)=\int_{-\infty}^r\int_{-\infty}^y
\bigl(\varphi(x)-f_\varepsilon(x)\bigr)\dif{x}\dif{y}
=\int_{-\infty}^r(r-x)
\bigl(\varphi(x)-f_\varepsilon(x)\bigr)\dif{x}.
\end{equation}
Then the difference
$$
F_\varepsilon(r)-F(r)=\max(0,r)-\int_{-\infty}^r(r-x)f_\varepsilon(x)\dif{x}
$$
is supported on $[-\varepsilon,\varepsilon]$ and bounded by
$\varepsilon$. Therefore, $\hat{F}_\varepsilon$ converges uniformly to
$\hat{F}$ as $\varepsilon\to0^+$.
Differentiating \eqref{eq:Fepsdef} twice, we have
\begin{equation*}
F_\varepsilon''(r)=\varphi(r)-f_\varepsilon(r),
\end{equation*}
and taking the Fourier transform of both sides we get
\begin{equation*}
-(2\pi t)^2\hat{F}_\varepsilon(t)=\hat{\varphi}(t)-\hat{f}(\varepsilon{t}).
\end{equation*}
Taking $\varepsilon\to0^+$ yields the desired identity.
\end{proof}

\begin{lemma}\label{lem:fudge}
Let $h_0(r)=\max(0,T-|r|)$ and $\varphi$, $F$ be as in Lemma~\ref{lem:fhat}. Then
\begin{equation*}
\bigl(h_0\ast\varphi-h_0\bigr)(r)= F(r-T)+F(r+T)+2F(r).
\end{equation*}
\end{lemma}
\begin{proof}
By the convolution theorem, we have
\begin{align*}
(h_0*\varphi-h_0)(r) &= \int_\R\hat{h}_0(t)\hat{\varphi}(t)\exp(2\pi i t r) \dif t -h_0(r)\\
&=\int_\R \hat{h}_0(t) \left[1-(2\pi t)^2 \hat{F}(t)\right] \exp(2\pi i t r)\dif t -h_0(r)\\
&=-\int_\R \hat{h}_0(t) (2\pi t)^2 \hat{F}(t) \exp(2\pi i t r)\dif t.
\end{align*}
A direct computation shows that $\hat{h}_0(t)=T^2\sinc^2(\pi Tt)$, where
$$
\sinc{t}:=\begin{cases}
\frac{\sin{t}}{t}&\text{if }t\ne0,\\
1&\text{if }t=0.
\end{cases}
$$
Thus,
\begin{align*}
(h_0*\varphi-h_0)(r)&=-\int_\R 4\sin^2(\pi t T)\hat{F}(t)\exp(2\pi itr)\dif t\\
&=\int_\R 2(\cos(2\pi t T)-1)\hat{F}(t)\exp(2\pi itr)\dif t\\
&=F(r-T)+F(r+T)+2F(r),
\end{align*}
as claimed.

\end{proof}

\begin{lemma}\label{lem:cont}
Let $\varphi$ be as in Lemma~\ref{lem:fhat}. Assume that
$\hat{\varphi}$ has compact support, so that $\varphi$ extends to an
entire function, and set
\begin{equation*}
V(r)=\int_{-\infty}^r\int_{-\infty}^y\varphi(x)\dif{x}\dif{y}
\quad\text{for }r\in\C.
\end{equation*}
Then, for any $T\in\R$, we have
\begin{equation*}
C\!\left(\frac{\cos(2\pi T t)}{2(\pi t)^2}\hat{\varphi}(t)\right)
=2V(-T)-2\Re V\!\left(\frac{i}{2}-T\right),
\end{equation*}
where $C(\cdot)$ is as defined in Proposition~\ref{prop:trace}.
\end{lemma}
\begin{proof}
Define $f(T):=C\!\left(\frac{\cos(2\pi T t)}{2(\pi t)^2}\hat{\varphi}(t)\right)$.
We first observe that $\frac{\cosh(\pi t)-1}{2(\pi t)^2}\hat{\varphi}(t)$ is
absolutely integrable, so by the Riemann--Lebesgue lemma,
$\lim_{T\to\infty}f(T)=0$.
Now differentiating twice with respect to $T$ gives us
\begin{equation*}
\begin{aligned}
f''(T)
&=-2\int_{\R}[\cosh(\pi t)-1]\hat{\varphi}(t)\cos(2\pi Tt)\dif t\\
&=2\int_{\R}\hat{\varphi}(t)\cos(2\pi Tt)\dif t-2\int_{\R}\cosh(\pi t)\hat{\varphi}(t)\cos(2\pi Tt)\dif t\\
&=\varphi(T)+\varphi(-T)-\frac{\varphi\!\left(\frac{i}{2}+T\right)+\varphi\!\left(\frac{i}{2}-T\right)+\varphi\!\left(-\frac{i}{2}+T\right)+\varphi\!\left(-\frac{i}{2}-T\right)}{2}.
\end{aligned}
\end{equation*}
On the other hand, if we start from
\begin{equation*}
g(T):=V(T)+V(-T)-\frac{V\!\left(\frac{i}{2}+T\right)+V\!\left(\frac{i}{2}-T\right)+V\!\left(-\frac{i}{2}+T\right)+V\!\left(-\frac{i}{2}-T\right)}{2}
\end{equation*}
and differentiate twice with respect to $T$, we get $g''(T)=f''(T)$.
Furthermore, Lemma \ref{lem:fhat} shows that $V(r)-\max(0,r)$ is even, and we get
\begin{align*}
g(T)&=2V(-T)+T-\Re V\!\left(\frac{i}{2}-T\right)-\Re V\!\left(\frac{i}{2}+T\right)\\
&=2V(-T)-2\Re V\!\left(\frac{i}{2}-T\right),
\end{align*}
and this vanishes in the limit as $T\to\infty$. Therefore, $f=g$.
\end{proof}

\begin{proposition}\label{prop:bound}
Let $\varphi$, $F$ and $V$ be as in
Lemmas~\ref{lem:fhat}--\ref{lem:cont},
and assume that $F(r)\ge0$ for $r\in\R$.
Let $\beta:\R\to\R$ be an even $C^2$ function of compact support with
$\beta(0)=1$, and let $h_2$ be the continous function with Fourier transform
\begin{equation*}
\hat{h}_2(t)=\frac{1-\beta(t)}{2(\pi t)^2}.
\end{equation*}
Define
\begin{equation*}
B=R(h_2)+(D-C)\!\left(\frac{\beta(t)}{2(\pi t)^2}\right)
\end{equation*}
and
\begin{equation*}
k(r)=\frac{r\tanh(\pi
r)}{12}+\frac{\frac{1}{8}+\frac{1}{3\sqrt{3}}
\cosh\!\left(\frac{\pi r}{3}\right)}
{\cosh(\pi r)}+\frac{\log(2\pi)-2\Re\psi(1+2ir)}{2\pi}.
\end{equation*}
Then for $T\geq 4$ we have
\begin{equation*}
\begin{aligned}
\int_0^T S(t)\dif t \leq &\int_{\R}[k(T+r)+k(T-r)]F(r)\dif r
-D\!\left(\frac{\cos(2\pi Tt)}{2(\pi t)^2}\varphi(t)\right)+B+C_0\\
&+\frac{\log T}{24\pi} -2\Re V\!\left(\frac{i}{2}-T\right),
\end{aligned}
\end{equation*}
with the constant $C_0$ defined as in Proposition~\ref{prop:M} above.
\end{proposition}
\begin{proof}
We wish to derive an upper bound for $\Tr(h_0)$.
As established by Proposition~\ref{prop:M}, we can
handle $M(h_0)$, and we now proceed by considering $R(h_0)=R(h_1)+R(h_0-h_1)$,
where $h_1$ majorizes $h_0$ and has a trace we can actually compute.

By hypothesis and Lemma~\ref{lem:fhat}, $F$ is even and non-negative, so
if we set $h_1=h_0\ast\varphi+2F$ then $h_1$ majorizes $h_0$. Further,
by Lemma~\ref{lem:fudge}, we have
\begin{equation*}
(h_1-h_0)(r)=F(r-T)+F(r+T).
\end{equation*}
Thus we can write
\begin{equation*}
\begin{aligned}
R(h_0-h_1)&=M(h_1-h_0)-\Tr(h_1-h_0)\leq M(h_1-h_0)\\
&=\int_\R k(r)\left[F(r-T)+F(r+T)\right]\dif r -2F(T).
\end{aligned}
\end{equation*}
Returning to $h_1$, by Lemma~\ref{lem:fhat} we have
\begin{equation*}
\hat{h}_1(t)=\hat{h}_0(t)\hat{\varphi}(t)+2\hat{F}(t)
=\frac{\sin^2(\pi T t)}{(\pi t)^2}\hat{\varphi}(t)+\frac{1-\hat{\varphi}(t)}{2(\pi t)^2}
=\frac{1}{2(\pi t)^2}-\frac{\cos(2\pi T t)}{2(\pi t)^2}\hat{\varphi}(t).
\end{equation*}
We are almost there, but we must cater for the $\frac{1}{2(\pi t)^2}$
term. To that end, we have
\begin{align*}
R(h_1)&=R(h_2)+R(h_1-h_2)
=R(h_2)+(D-C)\left[\frac{\beta(t)}{2(\pi t)^2}-\frac{\cos(2\pi T t)}{2(\pi t)^2}\hat{\varphi}(t)\right]\\
&=(C-D)\left[\frac{\cos(2\pi T t)}{2(\pi t)^2}\hat{\varphi}(t)\right]+B,
\end{align*}
where $B$ is the constant defined in the statement of the theorem.
Finally, Lemma~\ref{lem:cont} gives the formula for the continuous part.
\end{proof}

\subsection{Bound for large $T$}
For relatively small $T$, we can engineer things so that computing the discrete term is tractable. For large $T$, we use the following:
\begin{proposition}\label{prop:largeT}
Let the notation be as in Proposition~\ref{prop:bound} and assume that
$\hat{\varphi}(t)\ge0$ for $t\in\R$. Then for $T\ge4$,
\begin{equation}\label{eq:generalbound}
\begin{aligned}
\int_0^T S(t)\dif t \leq & \int_\R \left[k(T+r)+k(T-r)+2k(r)\right]F(r)\dif r\\
&-2\Re\!\left(V\!\left(\frac{i}{2}\right)+V\!\left(\frac{i}{2}-T\right)\right)+2B+C_0+\frac{\log T}{24\pi}.
\end{aligned}
\end{equation}
\end{proposition}
\begin{proof}
We bound the contribution from the discrete part trivially, using the assumption that $\hat{\varphi}$ is non-negative:
\begin{equation*}
\begin{aligned}
-D\!\left(\frac{\cos(2\pi T t)}{2(\pi t)^2}\hat{\varphi}(t)\right)&\leq
D\!\left(\frac{\hat{\varphi}(t)}{2(\pi t)^2}\right)\\
&=R(h_2-2F)+D\!\left(\frac{\beta(t)}{2(\pi t)^2}\right)+C\!\left(\frac{\hat{\varphi}(t)-\beta(t)}{2(\pi t)^2}\right)\\
&=-2R(F)+B+\int_\R\frac{\cosh(\pi t)-1}{2(\pi t)^2}\hat{\varphi}(t)\dif t\\
&\leq 2M(F)+B+\int_\R\frac{\cosh(\pi t)-1}{2(\pi t)^2}\hat{\varphi}(t)\dif t.
\end{aligned}
\end{equation*}
Lemma~\ref{lem:cont} with $T=0$ gives
\begin{equation*}
\int_\R\frac{\cosh(\pi t)-1}{2(\pi t)^2}\hat{\varphi}(t)\dif t
=2V(0)-2\Re V\!\left(\frac{i}{2}\right),
\end{equation*}
and we are done.
\end{proof}

\section{Bounding the constant $B$}\label{sec:B}

The first task is to derive a rigorous bound for the constant $B$, which we recall is defined via

\begin{equation*}
B=R(h_2)+(D-C)\!\left(\frac{\beta(t)}{2(\pi t)^2}\right).
\end{equation*}

\subsection{A suitable form for $\beta$}

Let $a,b,c>0$ and define
\begin{equation}\label{eq:beta}
\hat{\beta}(r)=c\sinc(\pi ar)^8(b^2-r^2).
\end{equation}
Choosing $c$ to make $\beta(0)=1$ now puts $\beta$ in the required form, with support $[-4a,4a]$. Also, we have

\begin{lemma}
Let $\hat{\beta}(r)$ be as defined above. Write
\begin{equation*}
k=\frac{c}{10080\pi^2a^8}.
\end{equation*}
Then for $t\in[0,a)$ we have
\begin{equation*}
\begin{aligned}
\beta(t)=k[&\pi^2b^2(4832a^7-3360a^5t^2+1120a^3t^4-280at^6+70t^7)\\
&-1680a^5+3360a^3t^2-2100at^4+735t^5];
\end{aligned}
\end{equation*}
for $t\in[a,2a)$ we have
\begin{equation*}
\begin{aligned}
\beta(t)=k[&\pi^2b^2(4944a^7-784a^6t-1008a^5t^2-3920a^4t^3+5040a^3t^4\\
&\;\;\;\;\;-2352a^2t^5+504at^6-42t^7)\\
&-504a^5-5880a^4t+15120a^3t^2-11760a^2t^3+3780at^4-441t^5];
\end{aligned}
\end{equation*}
for $t\in[2a,3a)$ we have
\begin{equation*}
\begin{aligned}
\beta(t)=-k[&\pi^2b^2(2224a^7-24304a^6t+38640a^5t^2-27440a^4t^3+10640a^3t^4\\
&\;\;\;\;\;-2352a^2t^5+280at^6-14t^7)\\
&+19320a^5-41160a^4t+31920a^3t^2-11760a^2t^3+2100at^4-147t^5];
\end{aligned}
\end{equation*}
for $t\in[3a,4a)$ we have
\begin{equation*}
\begin{aligned}
\beta(t)=k[&\pi^2b^2(32768a^7-57344a^6t+43008a^5t^2-17920a^4t^3+4480a^3t^4\\
&\;\;\;\;\;-672a^2t^5+56at^6-2t^7)\\
&+21504a^5-26880a^4t+13440a^3t^2-3360a^2t^3+420at^4-21t^5];
\end{aligned}
\end{equation*}
and for $t\in[4a,\infty)$ we have $\beta(t)=0$.
\end{lemma}
\begin{proof}
This is a messy but straightforward application of known Fourier transforms.
\end{proof}

A priori, there is no guarantee that Str\"ombergsson's list of $r_j$ is
complete. However, we will choose $b$ so that there are no unknown
$r_j\leq b$ (see \S\ref{subsec:results} below).  This makes $h_2(r_j)$ non-positive for any unknown $r_j$,
so that
\begin{equation*}
R(h_2)=\Tr^*(h_2)+\Tr^\dagger(h_2)-M(h_2)
\end{equation*}
where  $\Tr^*$ is the trace over the known $r_j$ and $\Tr^\dagger$ is
the trace over the rest. Thus
\begin{equation*}
B\leq \Tr^*(h_2)-M(h_2)+(D-C)\!\left(\frac{\beta(t)}{2(\pi t)^2}\right).
\end{equation*}

\subsection{Procedure}
We aim to rigorously compute 
\begin{equation*}
\Tr^*(h_2)-M(h_2)+(D-C)\!\left(\frac{\beta(t)}{2(\pi t)^2}\right).
\end{equation*}
\subsubsection{Computing $\Tr^*(h_2)$}
We have
\begin{equation*}
h_2(r)=\int_\R \hat{h}_2(t)\cos(2\pi t r) \dif t
=2\int_0^{4a}\frac{1-\beta(t)}{2(\pi t)^2} \cos(2\pi t r) \dif t+2\int_{4a}^\infty \frac{\cos(2\pi r t)}{2(\pi t)^2}\dif t.
\end{equation*}
The first integral we compute numerically for each $r_j$
in our database using Theorem~\ref{thm:Molin}. The second integral becomes
\begin{equation*}
\frac{2 r}{\pi}\left[\Si(8a\pi r)-\frac{\pi}{2}
+\frac{\cos(8a\pi r)}{8a\pi r}\right],
\end{equation*}
which again we compute for each of our known $r_j$.
The function $\Si$ above is the sine integral
\begin{equation*}
\Si(x)=\int_0^x\frac{\sin y}{y}\dif y.
\end{equation*}

\subsubsection{Computing $I(h_2)$}
We have
\begin{equation*}
I(h_2)=-\frac{1}{12\pi}\int_0^\infty\frac{\hat{h}_2'(t)}{\sinh(\pi t)}\dif t.
\end{equation*}
When $t$ is small, computing the $x_k$ for Theorem~\ref{thm:Molin} will
produce an interval that straddles zero. To avoid this, we work instead
with $\frac{\hat{h}_2'(t)}{t}$ and truncate the Taylor expansion of $\frac{t}{\sinh(\pi t)}$ after a few terms. The following lemma provides an error bound. 
\begin{lemma}
Let $N\geq 0$ and $|t|<\sqrt{(2N+4)(2N+5)}$. Then
\begin{equation*}
\left|\frac{\sinh{t}}{t}-\sum_{n=0}^N \frac{t^{2n}}{(2n+1)!}\right|\le\frac{t^{2N+2}(2N+4)(2N+5)}{(2N+3)!((2N+4)(2N+5)-t^2)}.
\end{equation*}
\end{lemma}
\begin{proof}
We majorize the tail of the Taylor expansion with the obvious geometric series.
\end{proof}

We can thus compute the integral over the intervals $[0,a]$, $[a,2a]$, $[2a,3a]$ and $[3a,4a]$ with little difficulty. This leaves computing the integral over $[4a,\infty)$ which we will truncate at some $t_0>4a$ using the following lemma:
\begin{lemma}
For $t_0>4a$ we have
\begin{equation*}
\left|\int_{t_0}^\infty \frac{\hat{h}_2'(t)}{\sinh(\pi t)}\dif t\right|
\le\left|\frac{1}{\pi^3 t_0^3} \log\tanh\left(\frac{\pi t_0}{2}\right)\right|.
\end{equation*}
\end{lemma}
\begin{proof}
With $t\geq t_0>4a$ we have
\begin{equation*}
\hat{h}_2'(t_0)=-\frac{1}{\pi^2 t_0^3},
\end{equation*}
and $\log\tanh(t/2)$ is an anti-derivative of $1/\sinh(t)$.
\end{proof}

Computing
\begin{equation*}
\int_{4a}^{t_0} \frac{1}{2(\pi t)^2 \sinh(\pi t)} \dif t
\end{equation*}
via Theorem~\ref{thm:Molin} still requires some care because of the pole
at $t=0$. As discussed in Appendix~\ref{app:comp} we can sidestep this by
taking the integrals over $[4a\alpha^n,4a\alpha^{n+1}]$ for
$n=0,\ldots,\lceil\log_\alpha(t_0/4a)\rceil$ for some $\alpha\in(1,3)$.
We took $\alpha=3-\frac{1}{128}$ and $t_0=4a\alpha^4$.

\subsubsection{Computing $E(h_2)$}
This is more straightforward. We need only the following auxiliary lemma.
\begin{lemma}\label{lem:E_tail}
Let $t_0>4a$. Then 
\begin{equation*}
\int_{t_0}^\infty \left(\frac{1}{8\cosh\pi t}+\frac{2\cosh\pi t}
{3+6\cosh 2\pi t}\right)\hat{h}_2(t)\dif t
\le\frac{7\exp(-\pi t_0)}{24\pi^3 t_0^2}\dif t.
\end{equation*}
\end{lemma}
\begin{proof}
We use the trivial estimate
\begin{equation*}
\int_{t_0}^\infty \left(\frac{1}{8\cosh\pi t}+\frac{2\cosh\pi t}
{3+6\cosh 2\pi t}\right)\hat{h}_2(t)\dif t
\le\int_{t_0}^\infty\left(\frac{\exp(-\pi t)}{4}
+\frac{\exp(-\pi t)}{3}\right)\frac{1}{2\pi^2t_0^2}\dif t.
\end{equation*}
\end{proof}
We again do four integrations covering $4a$ to
$4a\left(3-\frac{1}{128}\right)^4$ then use Lemma~\ref{lem:E_tail}
to bound the tail.

\subsubsection{Computing $P(h_2)$}
As a reminder, we have
\begin{equation*}
P(h)=\frac{\hat{h}(0)}{2\pi}\left(\log\frac{\pi}{2}+2\gamma\right)
-\frac{h(0)}{4}-\frac{1}{\pi}
\int_0^\infty\log\!\left(4\sinh\!\left(\frac{\pi t}{2}\right)\right)
\hat{h}'(t)\dif t.
\end{equation*}
Using Maple\textsuperscript{TM} we get
\begin{equation*}
\hat{h}_2(0)=c\frac{\pi^2a^2b^2-1}{6a^5\pi^4}
\end{equation*}
where $c$ is the constant in equation \ref{eq:beta}.
To compute $h_2(0)=\int_\R\hat{h}_2(t)\dif t$ we estimate the integrals
over $[0,a]$, $[a,2a]$, $[2a,3a]$ and $[3a,4a]$ using Theorem~\ref{thm:Molin}.
We then have
\begin{equation*}
\int_{4a}^\infty \hat{h}_2(t)\dif t = \frac{1}{8a\pi^2}.
\end{equation*}
To compute
\begin{equation*}
\int_0^a\log(4\sinh(\pi t/2))\hat{h}_2'(t)\dif t
\end{equation*}
we must handle the singularity at $t=0$. We compute
\begin{equation*}
\int_0^a(\log(4\sinh(\pi t/2))-\log t)\hat{h}_2'(t)\dif t+
\int_0^a(\log t)\hat{h}_2'(t)\dif t,
\end{equation*}
where the second integral can be computed analytically to yield
\begin{equation*}
\int_0^a(\log t)\hat{h}_2'(t)\dif t=c\frac{(195-130\pi^2a^2b^2)\log a +72\pi^2a^2b^2-115}{2880a^5\pi^4}.
\end{equation*}
Once past $4a$ we compute the three further integrals
\begin{equation*}
\sum_{n=0}^{2}\int_{4a\alpha^n}^{4a\alpha^{n+1}}\frac{\log\!\left(4\sinh\!\left(\frac{\pi t}{2}\right)\right)}{\pi^2 t^3}\dif t
\end{equation*}
and then bound the tail using the following lemma.
\begin{lemma}
Let $t_0>0$. Then
\begin{equation*}
\begin{aligned}
&\left|\int_{t_0}^\infty \log\!\left(4\sinh\!\left(\frac{\pi
t}{2}\right)\right)\hat{h}_2'(t)\dif t+(\frac{\pi t_0}{2}+\log
2)\hat{h}_2(t_0)+\frac{\pi}{2}\int_{t_0}^\infty \hat{h}_2(t)\dif t\right|\\
&\le\log(1-\exp(-\pi t_0))\hat{h}_2(t_0).
\end{aligned}
\end{equation*}
\end{lemma}
\begin{proof}
We use
\begin{equation*}
\log(4\sinh(t))=\log 2+t+\log(1-\exp(-2t))
\end{equation*}
and integration by parts.
\end{proof}

\subsubsection{Computing $D\!\left(\frac{\beta(t)}{2\pi^2t^2}\right)$}

This is straightforward other than the evaluations of $L(1,\chi_d)$. To
this end we produced a database containing class numbers $h(d)$ and
fundamental units ($u$ and $v$ such that $u^2-dv^2=\pm4$ with
$u,v\in\Z_{>0}$ and $v$ minimal). The database covered each fundamental
discriminant $d$ such that $dl^2=t^2-4$ with $t\in[3,10^5]$. We used
the PARI function \texttt{qfbclassno}\footnote{\texttt{quadclassuint} is
faster but the correctness of its results is conditional on GRH.} to compute the class numbers \cite{Batut2000} and the PQA algorithm due to Lagrange \cite{Niven2008} to compute the fundamental units. This then allows us to compute $L(1,\chi_d)$ for this range of $d$ rapidly and rigorously using Dirichlet's class number formula
\begin{equation*}
L(1,\chi_d)=\frac{2h(d)\log\!\left(\frac{u+v\sqrt{d}}{2}\right)}{\sqrt{d}}.
\end{equation*}
Since $\beta(t)$ is zero for $t\notin[-4a,4a]$, we can take
\begin{equation*}
a=\frac{1}{4\pi}\log\!\left(\frac{10^5+\sqrt{10^{10}-4}}{2}\right)=0.916169\ldots
\end{equation*}
without running out of pre-computed class group data.
We also need
\begin{equation*}
\frac{1}{\pi}\sum_{n=1}^\infty\frac{\Lambda(n)}{n}
h\!\left(\frac{\log n}{\pi}\right),
\end{equation*}
so for this value of $a$ we need to sum over primes and prime powers
$\leq 99\,991$.

\subsubsection{Computing $C\!\left(\frac{\beta(t)}{2\pi^2t^2}\right)$}
This reduces to the finite integral
\begin{equation*}
\int_{-4a}^{4a}\frac{(\cosh(\pi t)-1)\beta(t)}{2\pi^2t^2}\dif t.
\end{equation*}
The only issue here is the computation when $t$ is small,
when we resort to the series expansion of $t^{-2}(\cosh(t)-1)$
using the following lemma.
\begin{lemma}
Let $N\geq 0$ and $|t|<\sqrt{(2N+5)(2N+6)}$. Then
\begin{equation*}
\left|t^{-2}(\cosh(t)-1)-\sum_{n=0}^{N}\frac{t^{2n}}{(2n+2)!}\right|
\le\frac{t^{2N+2}(2N+5)(2N+6)}{((2N+5)(2N+6)-t^2)(2N+4)!}.
\end{equation*}
\end{lemma}
\begin{proof}
We majorize the tail of the series with the obvious geometric series.
\end{proof}

\subsection{Results}\label{subsec:results}
\subsubsection{$b=\frac{\sqrt{6\pi^2-1}}{2}$}
We will initially take $b=\frac{\sqrt{6\pi^2-1}}{2}$ since by
\cite[Theorem~11.4]{Iwaniec2002} there are no $r_j$ below this. Setting
$a=\frac{7505}{8192}$ and using no $r_j$ (so $\Tr^*(h_2)=0$) we get
\begin{equation*}
B\leq 0.272955804771976.
\end{equation*}
\subsubsection{$b=177.75$}
Once we have verified Corollary~\ref{cor:andreas} using the above bound
for $B$ (see Section~\ref{subsec:andreas} below), we can take $b=177.75$
and run the computation again. This time we get
\begin{equation*}
B\leq 0.2729558044747431.
\end{equation*}
The computation takes about twenty minutes on a single core, the time
being dominated by computing the trace of the known $r_j\le 177.75$.

\subsection{Improving the bound for $B$ further}
Even though the bound on $B$ derived above will more than suffice for
our immediate needs, it is possible to improve on it further. This could
be done by increasing the exponent $8$ used in the definition of $\beta$
(equation \eqref{eq:beta}). The true value of $B$ is more like
\begin{equation*}
B=0.2729558044747424323066650413\ldots,
\end{equation*}
but we do not prove that here.
 
\section{Verifying Theorem~\ref{th:turing}}\label{sec:comp}
Our verification proceeds in four stages. We will first prove the
theorem for $T\in[100,27\,400]$ via interval arithmetic and use this to
verify Corollary~\ref{cor:andreas}.  We will then rigorously verify that
Theorem~\ref{th:turing} holds for $T\in(1,100]$ by direct computation.
Next we will check, again using interval arithmetic, that the theorem
holds for $T\in[27\,400,10^6]$. Finally we will show that it
holds for all $T\ge10^6$.

First, however, we must define the function $\varphi$
used in Proposition~\ref{prop:bound}.
\subsection{A suitable form for $\hat\varphi$.}\label{sub:varphi}
In \cite{Littmann2006}, Littman shows that the unique best majorant
of exponential type $2\pi$ to $\max(0,r)$ is
\begin{equation*}
V_0(r)=\frac{\cos^2(\pi r)}{\pi^2}
\left[r\psi'\!\left(\frac{1}{2}-r\right)+1\right].
\end{equation*}
This would suggest that we set $\varphi(r)=X\varphi_0(Xr)$, where
$\varphi_0(r)=V_0''(r)$,
with Fourier transform
\begin{equation}\label{eq:varphihat}
\begin{aligned}
\hat{\varphi}_0(t)&=\left[1-\sum_{k=0}^\infty \frac{B_{k+2}\!\left(\frac{1}{2}\right)}{(k+1)!}\left(\frac{k+1}{k+2}-|t|\right)(2\pi i t)^{k+2}\right]\chi_{(-1,1)}(t)\\
&=\left(\frac{|t|}{\sinc(\pi t)}+\frac{(1-|t|)\cos \pi t}{\sinc^2(\pi t)}\right)\chi_{(-1,1)}(t),
\end{aligned}
\end{equation}
where $\chi_{(-1,1)}$ denotes the characteristic function of $(-1,1)$.
Unfortunately, this fails to be $C^4$ near $t=0$, which makes
the integral on the right-hand side of \eqref{eq:generalbound} diverge.
We add the following correction term which will allow us to
patch things up:
\begin{lemma}\label{lem:etahat}
Let $X,\delta>0$, and define
\begin{equation*}
\hat{\eta}_0(t)=\frac{\pi^2}{4+\pi^2}\left[\frac{2\pi^2}{3}(1-|t|)^3+4(1-|t|)(1-\cos(\pi t))-\frac{8}{\pi}\sin(\pi|t|)\right]\chi_{(-1,1)}(t)
\end{equation*}
and
\begin{equation*}
\hat{V}_1(t)=-\frac{\delta}{48\pi^2X^3}
\left[2\hat{\eta}_0\!\left(\frac{t}{\delta}\right)
+\hat{\eta}_0\!\left(\frac{t+X}{\delta}\right)
+\hat{\eta}_0\!\left(\frac{t-X}{\delta}\right)\right].
\end{equation*}
Then
\begin{equation*}
V_1(r)=\frac{\cos^2(\pi X r)}{\pi^2 X}
\left[-\frac{1}{12X^2r^2}+\frac{2\sinc^2(\pi \delta r)
+\sinc^2\left(\pi\delta r+\frac{\pi}{2}\right)
+\sinc^2\left(\pi\delta r-\frac{\pi}{2}\right)}{24(1+\frac{4}{\pi^2})X^2r^2}\right].
\end{equation*}
\end{lemma}
\begin{proof}
We have
\begin{equation*}
\eta_0(r)=\frac{1}{r^2}-\frac{2\sinc^2(\pi r)
+\sinc^2\!\left(\pi r+\frac{\pi}{2}\right)
+\sinc^2\!\left(\pi r-\frac{\pi}{2})\right)}{2\left(1+\frac{4}{\pi^2}\right)r^2},
\end{equation*}
and if we set
\begin{equation*}
V_1(r)=-\frac{\delta^2\cos^2(\pi X r)}{12\pi^2 X^3}\eta_0(\delta r)
=-\frac{\delta^2(\cos(2\pi X r)+1)}{24\pi^2 X^3}\eta_0(\delta r)
\end{equation*}
then
\begin{equation*}
\hat{V}_1(t)=-\frac{\delta}{48\pi^2 X^3}
\left[2\hat{\eta}_0\!\left(\frac{t}{\delta}\right)
+\hat{\eta}_0\!\left(\frac{t+X}{\delta}\right)
+\hat{\eta}_0\!\left(\frac{t-X}{\delta}\right)\right],
\end{equation*}
as required.
\end{proof}

We will need to bound the term involving the trigamma function. The following will suffice.
\begin{lemma}\label{lem:newtri}
For $z>0$ we have
\begin{equation*}
-z\psi'(\tfrac12+z)+1-\frac{1}{12z^2}\in\left[-\frac{7}{120z^4},0\right],
\end{equation*}
and for $z\in\R\setminus\{0\}$,
\begin{equation*}
\left|-iz\psi'(\tfrac12+iz)+1+\frac{1}{12z^2}\right|
\le\frac{112+105\pi}{3840z^4}.
\end{equation*}
Also, for $z=\sigma+it$ with $\sigma>0$ and $t\in\R$, we have
\begin{equation*}
\left|-z\psi'(\tfrac12+z)+1-\frac{1}{12z^2}\right|\le
\frac{7(\sigma+|t|)}{120\sigma^5}.
\end{equation*}
\end{lemma}
\begin{proof}
We use the multiplication formula to write
\begin{equation*}
\psi'(\tfrac{1}{2}+z)=4\zeta(2,2z)-\zeta(2,z)
\end{equation*}
for $-z\notin\Z_{\geq 0}$ where $\zeta(s,z)$ is the Hurwitz zeta-function. We now apply Euler--Maclaurin summation with
\begin{equation*}
g(t)=4(t+2z)^{-2}-(t+z)^{-2}
\end{equation*}
to get
\begin{equation*}
\sum_{n=0}^N g(n)+\int_N^\infty g(t)\dif t
-\frac{g(N)}{2}-\sum_{k=1}^K\frac{B_{2k}g^{(2k-1)}(N)}{(2k)!}
-\int_N^\infty\frac{B_{2K}(\{t\})}{(2K)!}g^{(2K)}(t)\dif t.
\end{equation*}
We now take $N=0$ and $K=2$. In the case $z>0$ the $k=2$ term is
$-\frac{7}{240 z^5}$ and the final integral is trivially less in
absolute terms than that. In the second case, the $k=2$ term is bounded
by $\frac{7}{240|z|^5}$ and the final integral contributes less than
\begin{equation*}
\frac{B_4}{24}\int_0^\infty\left(\frac{480}{(t^2+4z^2)^3}
-\frac{120}{(t^2+z^2)^3}\right)\dif t=\frac{7\pi}{256|z|^5}.
\end{equation*}
In the final case, the $k=2$ term yields $\frac{7}{240|z|^5}$. The integral can be bounded by $\frac{7}{240\sigma^5}$ and the result follows.
\end{proof}

We bound the $\sinc$ terms using \cite[Lemma~A.2]{Booker2015}:
\begin{lemma}\label{lem:sinc}
Let $r,\delta>0$ with $\delta^2 r^2\geq \frac{1}{12}$. Then
\begin{equation*}
2\sinc^2(\pi\delta r)+\sinc^2\!\left(\pi\delta r+\frac{\pi}{2}\right)
+\sinc^2\!\left(\pi\delta r-\frac{\pi}{2}\right)
\geq \frac{2}{\pi^2\delta^2 r^2}.
\end{equation*}
\end{lemma}
We can now combine the above to get:
\begin{proposition}\label{prop:varphi}
Let $\delta>0$,
$X\geq\delta\sqrt{\frac{7}{20}\left(1+\frac{4}{\pi^2}\right)}$. Take
$\hat{\varphi}_0$ and $\hat{\eta}_0$ to be as defined in
equation (\ref{eq:varphihat}) and Lemma~\ref{lem:etahat}, respectively. Now define
\begin{equation*}
\hat{\varphi}(t)=\hat{\varphi}_0\!\left(\frac{t}{X}\right)+\frac{\delta
t^2}{12
X^3}\left[2\hat{\eta}_0\!\left(\frac{t}{\delta}\right)+\hat{\eta}_0\!\left(\frac{t+X}{\delta}\right)+\hat{\eta}_0\!\left(\frac{t-X}{\delta}\right)\right].
\end{equation*}
Then $\hat{\varphi}$ meets all of the hypotheses of
Propositions~\ref{prop:bound} and \ref{prop:largeT}. In particular,
\begin{equation}\label{eq:Vdef}
\begin{aligned}
V(r)=\frac{\cos^2(\pi Xr)}{\pi^2 X}&\left[Xr\psi'\!\left(\frac{1}{2}-Xr\right)+1-\frac{1}{12X^2r^2}\right.\\
&\left.+\frac{2\sinc^2(\pi\delta r)+\sinc^2(\pi\delta r+\frac{\pi}{2})
+\sinc^2(\pi\delta r-\frac{\pi}{2})}{24(1+\frac{4}{\pi^2})X^2r^2}\right]
\end{aligned}
\end{equation}
is non-negative for $r\in\R$.
\end{proposition}
\begin{proof}
First we have the identity
\begin{equation*}
\cos^2(\pi z)\left[\psi'\!\left(\frac{1}{2}+z\right)+\psi'\!\left(\frac{1}{2}-z\right)\right]=\pi^2
\end{equation*}
which implies $V(z)=V(-z)+z$. Thus to establish positivity, we need
only consider
\begin{align*}
\frac{\cos^2(\pi Xr)}{\pi^2 X}\left[-Xr\psi'\!\left(\frac{1}{2}+Xr\right)+1-\frac{1}{12X^2r^2}
+\frac{2\sinc^2(\pi\delta r)+\sinc^2(\pi\delta r+\frac{\pi}{2})
+\sinc^2(\pi\delta r-\frac{\pi}{2})}{24(1+\frac{4}{\pi^2})X^2r^2}\right]
\end{align*}
for $r>0$. 
Now by Lemmas~\ref{lem:newtri} and \ref{lem:sinc} we get,
for $\delta^2 r^2\geq\frac{1}{12}$,
\begin{equation*}
V(-r)\ge\frac{\cos^2(\pi Xr)}{\pi^2
X}\left[\frac{1}{12\left(\pi^2+4\right)X^2\delta^2 r^4}-\frac{7}{120 X^4
r^4}\right],
\end{equation*}
which is non-negative providing $X\geq\delta\sqrt{\frac{7}{20}(\pi^2+4)}$,
as claimed. For $\delta^2 r^2<\frac{1}{12}$ we consider
\begin{equation*}
-(Xr)^3\psi'\!\left(\frac{1}{2}+Xr\right)+(Xr)^2
+\frac{2\sinc^2(\pi\delta r)
+\sinc^2(\pi\delta r+\frac{\pi}{2})
+\sinc^2(\pi\delta r-\frac{\pi}{2})}{24(1+\frac{4}{\pi^2})}
-\frac{1}{12}.
\end{equation*}
Since $-t^3\psi'\!\left(\frac{1}{2}+t\right)+t^2$ is an increasing
function of $t>0$, it suffices to consider
$X=c\delta$, where $c=\sqrt{\frac{7}{20}(\pi^2+4)}$.
Dividing by $(Xr)^2$ and writing $\delta r=t$, we obtain
\begin{equation*}
-ct\psi'(\tfrac{1}{2}+ct)+1
+\frac{f(t)-f(0)}{24(1+\frac4{\pi^2})c^2t^2},
\end{equation*}
where $f(t)=2\sinc^2(\pi t)+\sinc^2(\pi t+\tfrac{\pi}{2})
+\sinc^2(\pi t-\tfrac{\pi}{2})$.

By the Lagrange form of the error in Taylor's theorem, 
$(f(t)-f(0))/t^2=\frac12f''(u)$ for some $u\in[0,t]$.
We verify via interval arithmetic that $\frac12f''(t)>-5$ for
$t\in[0,1/\sqrt{12}]$ and that
\begin{equation*}
-ct\psi'(\tfrac{1}{2}+ct)+1
-\frac{5}{24(1+\frac4{\pi^2})c^2}>0
\end{equation*}
over the same interval.
\end{proof}

\subsection{Verifying Theorem~\ref{th:turing} for $T\in[100,27\,400]$}
For $T$ in this range, we proceed computationally. We chose $\delta=0.1$
and $X=2.55$ and kept them fixed throughout. If we computed all the terms
of Proposition~\ref{prop:bound} using interval arithmetic, we would be
forced to use a relatively narrow width, and the cost of computing 
\begin{equation*}
\int_\R [k(T+r)+k(T-r)]F(r)\dif r
\end{equation*}
using Theorem~\ref{thm:Molin} would be prohibitive. Fortunately, the part which determines the maximum width we can get away with is the $\cos(2\pi t T)$ term within $D(\cdot)$. The rest, including the expensive integrals, are much less susceptible, and we can use a width of between $2$ and $128$ depending on the size of $T$.

Having computed everything else with a relatively coarse interval, we then compute $D(\cdot)$ many times with a much narrower interval (typically about $2^{-8}$). Even here, most of the terms making up $D(\cdot)$ do not depend on $T$ so they can be pre-computed once and stored.

The entire computation coded in \texttt{C++} using interval arithmetic takes less than an hour on a $16$-core node of Bluecrystal Phase III \cite{ACRC2015}. Despite the foregoing, this time was still dominated by the rigorous computation of
\begin{equation*}
\int_\R [k(T+r)+k(T-r)]F(r)\dif r.
\end{equation*}

\begin{lemma}\label{lem:medium}
The computation shows that Theorem~\ref{th:turing} holds for $T\in[100,27\,400]$.
\end{lemma}

\subsection{Verifying Corollary~\ref{cor:andreas}}\label{subsec:andreas}
We now aim to partially validate our database. Writing $\overline{N}(t)=\frac{t^2}{12}
-\frac{2t}{\pi}\log\frac{t}{e\sqrt{\frac{\pi}2}}-\frac{131}{144}$,
we compute using interval arithmetic
\begin{equation*}
\int_0^{178} \overline{N}(t)\dif t \in (121\,643.023\,932,121\,643.023\,933).
\end{equation*}
Using Theorem~\ref{th:turing}, we compute
\begin{equation*}
\int_0^{178}S(t)\dif t < 0.398\,780
\end{equation*}
so we have
\begin{equation*}
\int_0^{178} N(t)\dif t < 121\,643.422\,713.
\end{equation*}
Also, using our (possibly incomplete) database of $r_j$ we can compute
\begin{equation*}
\int_0^{178}\Nmin(t)\dif t,
\end{equation*}
where $\Nmin$ is a minorant of $N$. Specifically, we have
\begin{equation*}
\int_0^{178}\Nmin(t)\dif t > 121\,643.206\,595.
\end{equation*}
Now if an $r_j$ were missing anywhere in $(0,178-(121\,643.422\,713-121\,643.206\,595))=(0,177.783\,882)$ then our revised value for $\int \Nmin (t)\dif t$ would exceed our upper bound for $\int N(t)\dif t$, a contradiction.

\subsection{Verifying Theorem~\ref{th:turing} for $T\in(1,100]$}
Once we are satisfied that our database of zeros is complete,
we can use it to confirm Theorem~\ref{th:turing} for small $T$ computationally.
\begin{lemma}\label{lem:small}
Theorem~\ref{th:turing} holds for $T\in(1,100]$.
\end{lemma}
\begin{proof}
The lemma holds trivially for $T\in(1,r_1)$ where $r_1=9.533\ldots$
is the location of the first zero as $\int_0^TN(t)\dif t=0$ and
$\int_0^T \overline{N}(t)\dif t +T\cdot E(T)$ is positive. Letting $N^+$
be the majorant of $N$ given by our list of $r_j$, we check that
\begin{equation*}
\int_0^T\bigl(N^+(t)-\overline{N}(t)\bigr)\dif t -T\cdot E(T) < 0
\end{equation*}
at each $r_j$ and divide the interval between the $r_j$'s into small
enough sub-intervals so that the inequality works there too. In all,
we checked a little under $500\,000$ intervals covering $(r_1,100]$ in
a little under $10$ seconds to verify the theorem. The nearest miss
in absolute terms was around $T=20.6862978$, where
\begin{equation*}
\int_0^T\bigl(N(t)-\overline{N}(t)\bigr)\dif t
-T\cdot E(T) = -9.8826792\ldots\times 10^{-8}.
\end{equation*}
\end{proof}

\subsection{Verifying Theorem~\ref{th:turing} for $T\in[27\,400,10^6]$}
Once $T$ is large enough, we can dispense with $D(\cdot)$ and
appeal to Proposition~\ref{prop:largeT} instead. We start with
$T\in[27\,400,27\,402]$ and check that Theorem~\ref{th:turing} holds, and
then move on to the next interval for $T$. After every $20$ iterations,
we try to double the width of the interval and continue. If at any
point we fail (presumably because our interval for $T$ grew too wide
too quickly) we halve the width of the interval and repeat. Coded in
\texttt{C++} using interval arithmetic, the computation takes less than
a couple of hours on a single core (and we could have parallelized it
trivially). By the end, the width of the interval had increased to $16\,384$ and we have:
\begin{lemma}\label{lem:large}
Theorem~\ref{th:turing} holds for $T\in[27\,400,10^6]$.
\end{lemma}

\subsection{Verifying Theorem~\ref{th:turing} for $T\geq 10^6$}
We now look at each of the contributions to our bound for $\int S(t)\dif
t$ coming from Proposition~\ref{prop:largeT}. We will need several
preparatory lemmas. We assume the notation and hypotheses of
Propositions~\ref{prop:bound} and \ref{prop:varphi}, and we fix
$\delta=0.842$ throughout.

\begin{lemma}\label{lem:fhat0}
We have
\begin{equation*}
\hat{F}(0)=\frac{3(\pi^2+4)X-2\pi^2\delta}{72X^3(\pi^2+4)}.
\end{equation*}
\end{lemma}
\begin{proof}
We have
\begin{equation*}
\hat{F}(0)=\lim_{t\rightarrow 0}\frac{1-\hat{\varphi}(t)}{(2\pi t)^2}.
\end{equation*}
Then taking $t$ positive and sufficiently close to zero eliminates the $\hat{\varphi}((t\pm X)/\delta)$ terms and we have
\begin{equation*}
\hat{\varphi}(t)=\hat{\varphi}_0\!\left(\frac{t}{X}\right)+\frac{\delta
t^2}{6X^3}\hat{\eta}_0\!\left(\frac{t}{\delta}\right).
\end{equation*}
Now we have
\begin{equation*}
\hat{\varphi}_0\!\left(\frac{t}{X}\right)=1-\frac{1}{6}\frac{\pi^2}{X^2}t^2+\bigoh(t^3)
\quad\text{and}\quad
\frac{\delta t^2}{6X^3}\hat{\eta}_0\!\left(\frac{t}{\delta}\right)=\frac{\pi^4\delta}{9(\pi^2+4)X^3}t^2+\bigoh(t^3),
\end{equation*}
so that
\begin{equation*}
1-\hat{\varphi}(t)=\frac{\pi^2(3(\pi^2+4)X-2\pi^2\delta)}
{18(\pi^2+4)X^3}t^2+\bigoh(t^3).
\end{equation*}
The result follows on taking $t\rightarrow 0$.
\end{proof}

\begin{lemma}\label{lem:V_new}
For $r\in\R\setminus\{0\}$, we have
\begin{equation*}
F(r)\le\frac1{400X^3r^4}.
\end{equation*}
\end{lemma}
\begin{proof}
We use Lemma~\ref{lem:newtri} and the estimate
\begin{equation*}
2\sinc^2(\pi\delta r)+\sinc^2(\pi\delta r+\pi/2)
+\sinc^2(\pi\delta r-\pi/2)\le\frac{24(\pi^2+4)}{400r^2}.
\end{equation*}
\end{proof}

\begin{lemma}\label{lem:k}
We have
\begin{equation*}
k(r)\le\begin{cases}
\frac{4}{5}&\text{if }|r|<2,\\
\frac{|r|}{12}&\text{if }|r|\ge2.
\end{cases}
\end{equation*}
\begin{proof}
Since $k$ is an even function of $r$, we need only consider $r\geq 0$. We first check that $k(r)\leq\frac{4}{5}$ for $r\in[0,2]$ by interval arithmetic. Then since
\begin{equation*}
\frac{r\tanh(\pi r)}{12}\leq \frac{r}{12}
\end{equation*}
for all $r\geq 0$, we need only check that at $r=2$
\begin{equation*}
\frac{\frac{1}{8}+\frac{1}{3\sqrt{3}}\cosh\left(\frac{\pi r}{3}\right)}{\cosh(\pi r)}+\frac{\log 2\pi -2\Re\psi(1+2ir)}{2\pi}<0
\end{equation*}
 and observe that the $\Re\psi(1+2ir)$ term is an increasing function of $r$ whereas the $\cosh$ term is decreasing.
\end{proof}
\end{lemma}

\begin{lemma}\label{lem:kTr_new}
For $T\ge10$ we have
\begin{align*}
\int_\R\left[k(T+r)+k(T-r)\right]F(r)\dif r&\le
\frac{T}{144X^2}-\frac{T\pi^2\delta}{216X^3(\pi^2+4)}\\
&+\frac{5T^4+15T^3+258T^2+20T+264}{18000X^3(T-2)^3(T+2)^3}.
\end{align*}
\end{lemma}
\begin{proof}
Since $F$ is even, we write
\begin{equation*}
\int_\R\bigl[k(T+r)+k(T-r)\bigr]F(r)\dif r =2\int_\R k(T+r)F(r)\dif r
\end{equation*}
and then split the integral into three pieces:
\begin{equation*}
2\int_{-\infty}^{-T-2}k(T+r)F(r)\dif r
\le2\int_{-\infty}^{-T-2}\frac{|T+r|}{12}\frac{\dif{r}}{400X^3r^4}
=\frac1{3600X^3}\left(\tfrac14(T+2)^{-2}+(T+2)^{-3}\right),
\end{equation*}
\begin{equation*}
2\int_{-T-2}^{-T+2}k(T+r)F(r)\dif r
\le2\int_{-T-2}^{-T+2}\frac{4}{5}\cdot\frac{\dif{r}}{400X^3r^4}
=\frac2{375X^3}\frac{3T^2+4}{(T-2)^3(T+2)^3},
\end{equation*}
and
\begin{align*}
2\int_{-T+2}^\infty k(T+r)F(r)\dif r
&\le2\int_{-T+2}^\infty \frac{T+r}{12}F(r)\dif r
=2\int_{-T+2}^\infty \frac{T}{12}F(r)\dif r
+2\int_{-T+2}^\infty\frac{r}{12}F(r)\dif{r}\\
&\le\frac{T}{6}\hat{F}(0)
+2\int_{T-2}^\infty\frac{r}{12}\frac{\dif{r}}{400X^3r^4}\\
&=\frac{3(\pi^2+4)X-2\pi^2\delta}{72X^3(\pi^2+4)}\frac{T}{6}
+\frac1{4800X^3}(T-2)^{-2}.
\end{align*}
\end{proof}

\begin{lemma}\label{lem:krFr_new}
We have
\begin{equation*}
2\int_\R k(r)F(r)\dif{r}\le\frac1{15X^2}.
\end{equation*}
\end{lemma}
\begin{proof}
We have
\begin{align*}
2\int_\R k(r)F(r)\dif{r}
&\le2\int_{-2}^2\frac{4}{5}F(r)\dif{r}
+4\int_2^\infty\frac{r}{12}\frac{\dif{r}}{400X^3r^4}\\
&\le\frac85\hat{F}(0)+\frac1{9600X^3}
=\frac1{15X^2}-\frac{2\pi^2\delta}{45(\pi^2+4)X^3}+\frac1{9600X^3}\\
&\le\frac1{15X^2}.
\end{align*}
\end{proof}

\begin{lemma}\label{lem:Vi2_new}
We have
\begin{equation*}
-2\Re V\!\left(\frac{i}{2}\right)\leq
\cosh^2\!\left(\frac{\pi X}{2}\right)\left[0.09752X^{-3}+0.3731X^{-5}\right].
\end{equation*}
\end{lemma}
\begin{proof}
We have
\begin{equation*}
\frac{\cos^2\!\left(\frac{\pi X i}{2}\right)}{\pi^2X}
=\frac{\cosh^2\!\left(\frac{\pi X}{2}\right)}{\pi^2 X}
\end{equation*}
and (from Lemma~\ref{lem:newtri})
\begin{equation*}
\left|-\frac{i}{2}X\psi'\!\left(\frac{1}{2}+\frac{iX}{2}\right)+1
-\frac{1}{12X^2(i/2)^2}\right|\leq \frac{112+105\pi}{240X^4}.
\end{equation*}
Setting $\delta=0.842$, the terms of \eqref{eq:Vdef} involving $\delta$
can be computed directly, and this yields the claimed inequality.
\end{proof}

\begin{lemma}\label{lem:revi2T_new}
We have
\begin{equation*}
\left|2\Re V(i/2-T)\right|\leq
\frac{\cosh^2(\pi X/2)}{X^5T^4}(0.07914+0.01183 T^{-1}).
\end{equation*}
\end{lemma}
\begin{proof}
For $x,y\in\R$ we have
\begin{align*}
|\cos(x+iy)|\le\sqrt{2}\cosh(y)
\quad\text{and}\quad
|\sin(x+iy)|\le\sqrt{2}\cosh(y),
\end{align*}
so that
\begin{equation*}
\left|\cos^2\!\left(\pi
X\left(\tfrac{i}{2}-T\right)\right)\right|\leq2\cosh^2\!\left(\frac{\pi X}{2}\right).
\end{equation*}
We also have
\begin{equation*}
\left|2\sinc^2\!\left(\pi\delta\left(\tfrac{i}{2}-T\right)\right)
+\sinc^2\!\left(\pi\delta\left(\tfrac{i}{2}-T+\tfrac{1}{2}\right)\right)
+\sinc^2\!\left(\pi\delta\left(\tfrac{i}{2}-T-\tfrac{1}{2}\right)\right)\right|
\le\frac{8\cosh^2\!\left(\frac{\pi\delta}{2}\right)}{\pi^2\delta^2T^2}
\end{equation*}
and Lemma~\ref{lem:newtri} gives us
\begin{equation*}
\left|-(XT-\tfrac{iX}{2})
\psi'\!\left(\tfrac{1}{2}+XT-\tfrac{iX}{2}\right)
+1-\frac{1}{12\left(XT-\tfrac{iX}{2}\right)^2}\right|
\leq\frac{7}{120X^4T^4}+\frac{7}{240X^4T^5}.
\end{equation*}
The result follows on setting $\delta=0.842$.
\end{proof}

We can now establish our bound for $T\geq 10^6$.
\begin{lemma}\label{lem:vlarge}
Let $S(t)$ and $E(T)$ be as defined in Theorem~\ref{th:turing}.
Then for $T\geq 10^6$ we have
\begin{equation*}
\begin{aligned}
\frac{1}{T}\int_0^T S(t)\dif t-E(T) \le 0.
\end{aligned}
\end{equation*} 
\end{lemma}
\begin{proof}
\begin{align*}
\int_0^TS(t)\dif{t}&\le
\int_\R \left[k(T+r)+k(T-r)+2k(r)\right]F(r)\dif r
-2\Re\!\left(V\!\left(\frac{i}{2}\right)+V\!\left(\frac{i}{2}-T\right)\right)\\
&\hspace{1cm}+2B+C_0+\frac{\log T}{24\pi}\\
&\le\frac{T}{144X^2}-\frac{T\pi^2\delta}{216(\pi^2+4)X^3}
+\frac{5T^4+15T^3+258T^2+20T+264}{18000X^3(T-2)^3(T+2)^3}+\frac1{15X^2}\\
&\hspace{1cm}+2B+C_0+\frac{\log T}{24\pi}\\
&\hspace{1cm}+\cosh^2\!\left(\frac{\pi X}{2}\right)
\bigl[0.09752X^{-3}+0.3731X^{-5}\bigr]\\
&\hspace{1cm}+\frac{\cosh^2(\pi X/2)}{X^5T^4}
\bigl[0.07914+0.01183T^{-1}\bigr].
\end{align*}
Assuming that $T\ge10^6$, this is at most
\begin{align*}
&\frac{T}{144X^2}-\frac{T\pi^2\delta}{216(\pi^2+4)X^3}
+\frac1{15X^2}
+2B+C_0+\frac{\log T}{24\pi}\\
&\hspace{1cm}+\cosh^2\!\left(\frac{\pi X}{2}\right)
\bigl[0.09753X^{-3}+0.3732X^{-5}\bigr].
\end{align*}
Now set $X=\frac1\pi\log(T/5)$. Then, since
$X\ge\frac1\pi\log(200000)$, we have
\begin{align*}
&\frac1{15X^2}+2B+C_0
+\frac{\pi X+\log5}{24\pi}\\
&\hspace{1cm}+\cosh^2\!\left(\frac{\pi X}{2}\right)
\bigl[0.09753X^{-3}+0.3732X^{-5}\bigr]\\
&\le 0.0308\exp(\pi X)X^{-3}=0.00616TX^{-3}.
\end{align*}
Since $\pi^2\delta/216(\pi^2+4)>0.00277$, we get
$$
\frac1T\int_0^TS(t)\dif{t}\le\frac{\pi^2}{144\log^2(T/5)}
+\frac{0.1052}{\log^3(T/5)}.
$$
Note that
\begin{align*}
&\frac{144}{\pi^2}\log^3(T/5)\log^3(T)\left(
\frac{\pi^2}{144\log^2(T/5)}+\frac{0.1052}{\log^3(T/5)}-E(T)\right)\\
&=\log(T/5)\log^3(T)-\log(T)\log^3(T/5)
+\frac{144}{\pi^2}0.1052\log^3(T)-6.59125\log^3(T/5)\\
&=2\log(5)\log(T/5)\log(T)\log(T/\sqrt{5})
+\frac{144}{\pi^2}0.1052\log^3(T)-6.59125\log^3(T/5).
\end{align*}
This is a cubic polynomial in $\log{T}$, and it is
straightforward to see that it is negative for $T\ge10^6$.
\end{proof}

\appendix

\section{Rigorous quadrature}\label{app:comp}
Part of the proof of Theorem~\ref{th:turing} will involve computational
verification of certain statements involving integrals. We will
control rounding errors by employing interval arithmetic\footnote{We
use Johansson's \texttt{ARB} package \cite{Johansson2013}.} to obtain
rigorous results and we will rely on the following theorem to perform
rigorous quadrature:
\begin{theorem}[Molin]\label{thm:Molin}
Let $f$ be a holomorphic function on $D(0,2)=\{z\in\C:|z|\leq 2\}$.
Then for all $n\geq 1$ we have
\begin{equation*}
\left|\int_{-1}^1f(x)\dif x -\sum_{k=-n}^n
a_kf(x_k)\right|\leq\exp\!\left(4-\frac{5n}{\log(5n)}\right)\sup_{D(0,2)}|f|,
\end{equation*}
where $h=\frac{\log(5n)}{n}$, $a_k=\frac{h\cosh(kh)}{\cosh(\sinh(kh))^2}$ and $x_k=\tanh(\sinh(kh))$.
\begin{proof}
See \cite[Theorem~1.1]{Molin2010a}, which is derived from
\cite[Theorem~2.10]{Molin2010}.
\end{proof}
\end{theorem}

Under the conditions of the above theorem, the supremum of $|f(z)|$ with $z\in D(0,2)$ lies somewhere on the boundary $|z|=2$. If we are unable (or unwilling) to bound this supremum analytically, we can ``cheat'' and divide $[0,1]$ into sufficiently small intervals $\theta_n$, compute using interval arithmetic each
\begin{equation*}
|f(2\exp(2\pi i\theta_n))|
\end{equation*}
and take the maximum.

To apply this theorem in practice, consider estimating
\begin{equation*}
\int_{t_0}^{t_1} f(t)\dif t
\end{equation*}
where $t_1>t_0$. We can rescale this to
\begin{equation*}
\frac{t_1-t_0}{2}\int_{-1}^{1}f\!\left(\frac{2w-t_1-t_0}{t_1-t_0}\right) \dif w
\end{equation*}
and appeal to Theorem~\ref{thm:Molin} directly providing the rescaled function has no poles inside the circle $|z|=2$. The most common awkward situation is where $f$ does have a single pole on the real line at $t=\rho$ and $t_1>t_0>\rho$. In this case, we are still in the clear so long as $t_1<3t_0-2\rho$ which may force us to perform the integral piecewise to cover the required interval. 

\bibliographystyle{amsplain}
\bibliography{turing6}
\end{document}